\documentclass[a4paper,11pt]{article}
\usepackage{amsmath,amssymb,amsthm}
\usepackage{cases}
\usepackage{a4wide}
\usepackage[pdftex]{graphicx}
\usepackage[hang,small,bf]{caption}
\usepackage[subrefformat=parens]{subcaption}
\usepackage{here}
\usepackage{comment}
\usepackage{ulem}
\usepackage[pdftex]{color}

\newtheorem{prop}{Proposition}
\newtheorem{thm}{Theorem}
\newtheorem{lem}{Lemma}
\newtheorem{rem}{Remark}
\newtheorem{cor}{Corollary}

\newtheorem{dfn}{Definition}

\newcommand{\dH}[1]{d\mathcal{H}^{#1}}

\newcommand{\wt}[1]{\widetilde{#1}}
\newcommand{\pOmega}[0]{\partial\Omega}
\newcommand{\closure}[1]{\overline{#1}}
\newcommand{\extR}[0]{\mathbb{R}\cup\{\infty\}}
\newcommand{\subdiff}[2]{\partial #1 (#2)}

\newcommand{\TV}[2]{\int_{#1}|\nabla{#2}|}
\newcommand{\iter}[3]{{#1}^{(#2)}_{#3}}
\newcommand{\argmin}{\mathop{\rm arg~min}\limits}
\newcommand{\deriv}[1]{\frac{d}{d #1}}

\newcommand{\partialDeriv}[4]{(#1_{#2})_{#3,#4}}
\newcommand{\script}[2]{\mathcal{#1}_{#2}}
\newcommand{\chambolleEnergy}[2]{E^{#1}_#2}
\newcommand{\geodis}[1]{d_{\Omega,#1}}

\title{\textbf{On a minimizing movement scheme for mean curvature flow with prescribed contact angle in a curved domain and its computation}}
\author{Tokuhiro Eto \thanks{Graduate School of Mathematical Sciences, The University of Tokyo, Komaba 3-8-1, Meguro, Tokyo 153-8914, Japan. E-mail:tokuhiro\_eto@yahoo.co.jp} \and Yoshikazu Giga \thanks{Graduate School of Mathematical Sciences, The University of Tokyo, Komaba 3-8-1, Meguro, Tokyo 153-8914, Japan. E-mail:labgiga@ms.u-tokyo.ac.jp}}
\date{\today}
\begin{document}
\begin{comment}
  A manuscript to be submitted to a journal.
\end{comment}
\maketitle

\begin{abstract}
We introduce a capillary Chambolle type scheme for mean curvature flow with prescribed contact angle.
 Our scheme includes a capillary functional instead of just the total variation.
 We show that the scheme is well-defined and has consistency with the energy minimizing scheme of Almgren-Taylor-Wang type.
 Moreover, for a planar motion in a strip, we give several examples of numerical computation of this scheme based on the split Bregman method instead of a duality method.
\end{abstract}


\section{Introduction}
% 平均曲率流方程式の問題背景・意義
In this study, we consider the mean curvature flow equation with prescribed contact angle condition of the form
\begin{equation}
    \begin{cases}
        V = -\operatorname{div}_{\Gamma_t}\mathbf{n}\ \ \mbox{on}\ \ \Gamma_t \cap \Omega\ \ \mbox{for}\ \ t\geq 0,\\
        \angle(\mathbf{n},\mathbf{n}_\Omega) = \theta(t,\cdot)\ \ \mbox{on}\ \ \partial\Gamma_t\cap\pOmega\ \ \mbox{for}\ \ t\geq 0,
    \end{cases}\tag{MCFB}\label{MCFB}
\end{equation}
where $\Omega\subset\mathbb{R}^d$ is a smooth bounded domain and $\mathbf{n}_\Omega$ is the unit normal velocity
vector field on $\pOmega$.
 Here, $\{\Gamma_t\}_t$ is a time evolving hypersurface to be determined and $\mathbf{n}$ represents the outward unit normal vector field of $\Gamma_t$;
 $V$ denotes the velocity of $\Gamma_t$ in the direction of $\mathbf{n}$, which is the outward unit normal vector to $\Gamma_t$, and $\theta$ is a given function on $[0,T]\times\pOmega$ 
that describes the contact angle between $\Gamma_t$ and $\pOmega$ for each $t\geq 0$.
 Here, $\operatorname{div}_{\Gamma_t}$ denotes the surface divergence so that $-\operatorname{div}_{\Gamma_t}\mathbf{n}$ becomes the ($d-1$ times) mean curvature of $\Gamma_t$ in the direction of $\mathbf{n}$.

For the mean curvature flow equation in $\mathbb{R}^d$, Almgren, Taylor and Wang \cite{AlmregTaylorWang1993} introduced a time discrete approximation of the solution which is often called the Almgren-Taylor-Wang scheme;
 a similar scheme is given by Luckhaus and Sturzenhecker \cite{LS}.
 However, one has to minimize a non-convex problem for each time step.
 In \cite{Chambolle2004}, Chambolle introduced another scheme which is based on a strict convex problem and the solution of each step chooses one of minimizers of the Almgren-Taylor-Wang's functional.

One of goals of this paper is to extend Chambolle's scheme to the problem \eqref{MCFB} which includes the prescribed contact angle condition.
 We call our scheme the capillary Chambolle type scheme.
 We shall show that the scheme is well-defined, and it chooses one of minimizers of the corresponding Almgren-Taylor-Wang's functional. % the Almgrenではない（4行上と異なる）
 We shall explain them more explicitly.

% 先行研究の紹介
% この論文で考えること、得られた結果
Let us first introduce the capillary  Almgren-Taylor-Wang functional.
 Given a Caccioppoli set $F_0\subset\Omega$, the capillary Almgren-Taylor-Wang functional is defined by
\begin{equation}\label{eq:A_b_def}
    \script{A}{\beta}(F,F_0,\lambda) := \script{C}{\beta}(F) + \lambda\int_{F\triangle F_0}\operatorname{dist}(\cdot,\partial F_0)\,dx
\end{equation}
for each Caccioppoli set $F\subset\Omega$, where $F\Delta E_0=(E_0\backslash F)\cup(F\backslash E_0)$.
 Here, $\script{C}{\beta}$ denotes the capillary functional defined by
\begin{equation*}
    \script{C}{\beta}(F) := \int_{\Omega}|\nabla\chi_{F}| + \int_{\partial\Omega}\beta\gamma\chi_F\,\dH{d-1},
\end{equation*}
where $\mathcal{H}^{d-1}$ denotes the $d-1$ dimensional Hausdorff measure.
 The function $\beta\in L^\infty(\partial\Omega)$ will be formally taken as $\beta=\cos\theta$ on the intersection of $\partial\Omega$ and the boundary of $F$.
 One observes that the functional defined in \eqref{eq:A_b_def} is reduced to the functional introduced by Almgren, Taylor and Wang \cite{AlmregTaylorWang1993} when $\beta\equiv0$.
 The capillary scheme is as follows.
 Let $h=1/\lambda$ be a given time step.
 For a given Caccioppoli set $E_0$ find a minimizer $(E_0)_h$ which minimizes $E\mapsto\mathcal{A}_\beta(E,E_0,\lambda)$.
 We repeat this process and find a sequence of sets which is expected to approximate the solution of \eqref{MCFB}.
 However, the functional $\mathcal{A}_\beta$ is not convex so it is a priori not easy to find its minimizer.
 To overcome this inconvenience, we introduce the capillary Chambolle type scheme.
 Let $\Omega$ be a smooth bounded domain. Then, for each $u\in L^2(\Omega)$, we define

\begin{equation}\label{eq:E_b_def}
    \chambolleEnergy{\beta}{h}(u) := C_\beta(u) + \frac{1}{2h}\int_\Omega(u - \geodis{E_0})^2\,dx.
\end{equation}
Here, $h>0$ is a time step that discretizes an interval $[0,T]$ for some time horizon $T>0$; $\geodis{E_0}$ denotes the geodesic signed distance function to $E_0\subset\Omega$ with respect to $\Omega$.
The capillary functional $C_\beta$ is defined by

\begin{equation}\label{eq:Capillary_function}
    C_\beta(u) := \int_\Omega |\nabla u| + \int_{\pOmega}\beta\gamma u\,\dH{d-1}.
    \tag{CF}
\end{equation}

The functional $E_h^\beta$ is a strictly convex functional on $L^2(\Omega)$, so there exists a unique minimizer once we know $E_h^\beta$ is lower semi-continuous in $L^2(\Omega)$.
 In fact, Modica \cite{Modica1987} proved that $C_\beta(u)$ is lower semi-continuous in $L^1(\Omega)$ when $\Omega$ is a $C^2$ bounded domain and $\|\beta\|_\infty\leq1$.
 The assumption $\|\beta\|_\infty\leq1$ is natural since it is given as a cosine functional.
 As well-known, the condition $\|\beta\|_\infty\leq1$ is a necessary condition for the lower semi-continuity;
 see e.g.\ \cite{GNRS}.
 Although his lower semi-continuity result is enough for our purpose, we give a proof of the semi-continuity in $L^1(\Omega)$ which works for any uniformly $C^2$ domain not necessarily bounded.
 We first prove it for $\|\beta\|_\infty<1$ and using an argument by contradiction as in \cite[Proof of Lemma 2]{CaffarelliMellet2007} to prove the case $\|\beta\|_\infty=1$.
 See Proposition \ref{prop:lsc_d_dim_capillary}.

We shall explain the capillary Chambolle type scheme.
Given a set $E\subset\Omega$, consider the minimizing problem of the energy $E_h^\beta(u)$
in the Lebesgue space $L^2(\Omega)$.
 Since $E_h^\beta$ with $E_0=E$ is lower semi-continuous and convex in the topology $L^2(\Omega)$, 
we see that $E_h^\beta$ has a unique minimizer $w^h_E\in L^2(\Omega)$.
 We set
\[
	T^h(E) = \left\{ x \in \Omega \Bigm|
	w_E^h(x) \leq 0 \right\}.
\]
It turns out that this $T^h(E)$ is a minimizer of the capillary Almgren-Taylor-Wang functional.
 In the case $\beta\equiv0$, this was proved by Chambolle \cite[Proposition 2.2]{Chambolle2004}.
 More precisely, we have
\begin{thm} \label{Main1}
    Let $\Omega$ be a bounded $C^2$ domain in $\mathbb{R}^d$. Assume that 
    $\beta\in L^\infty(\pOmega)$ satisfies $\|\beta\|_\infty\leq 1$ and $\int_{\pOmega}\beta\dH{d-1} = 0$.
    Then, for any $E_0\subset\Omega$, $T_h^\beta(E_0)$ is a minimizer of $\script{A}{-\beta}(\cdot,E_0,1/h)$.
\end{thm}

Moreover, characterization of the subdifferential of the capillary functional $C_\beta$ seems important because Chambolle 
used this characterization in implementation of his scheme. Though we do not adopt this direction in numerical experiment, we state its rigorous form.
\begin{thm} \label{Main2}
    Let $\Omega$ be a bounded $C^2$ domain in $\mathbb{R}^d$.
 For $(u,p)\in L^2(\Omega)\times L^2(\Omega)$, $p \in\partial C_\beta(u)$ if and only if there exists $z\in L^\infty(\Omega;\mathbb{R}^d)$ with $\operatorname{div}z\in L^2(\Omega)$ such that $p = -\operatorname{div}{z}$ in $\mathcal{D}'(\Omega)$, $\beta = -[z\cdot\nu]$ $\mathcal{H}^{d-1}$-a.e.\ on $\partial\Omega$ and $\int_\Omega(z,Du)=\int_\Omega|\nabla u|$ with $\|z\|_\infty \leq 1$, where $(z,Du)$ denotes the Anzellotti pair.
\end{thm}
Here $\mathcal{D}'(\Omega)$ is the space of Schwartz's distributions and $[z\cdot\nu]$ is the normal trace.
 We prove Theorem \ref{Main1} by a duality argument for positively homogeneous function due to Alter; see e.g.\ \cite{CaffarelliMellet2007}.
 However, $C_\beta(u)$ may not be positive for some $u$.
 We add a linear functional and characterize its subdifferential.
 Unfortunately, the characterization by a duality argument is more involved because of addition of a linear function.

We also give a numerical simulation of our scheme.
 In other words, for given initial data $E_0$, we define a discrete solution
\[
    E^h(t) := T_h^{\lfloor\frac{t}{h}\rfloor}(E_0).
\]
In \cite{Chambolle2004_TV}, Chambolle gave a way to calculate the minimizer of $E_h^\beta$ with $\beta\equiv0$ based on duality.
 Although this idea applies several problems including higher-order total variation flow \cite{GMR}, we do not use his approach.
 Instead, we adapt a split Bregman method as applied by \cite{ObermanOsherTakeiTsai2011} to calculate a planar crystalline curvature flow.
 This method was introduced by Goldstein and Osher \cite{GoldsteinOsher2009} to calculate energy minimizer including total variation.
 It applies the fourth order total variation flow \cite{GU}.
 Our domain $\Omega$ is a strip, and we calculate various examples including translative soliton \cite{AltschulerWu1993}.
 Since in Theorem \ref{Main1}, we are forced to assume that the average of $\beta$ is equal to zero, at each time step, we redefine $\beta$ outside contact points.

Let us remark a few related preceding works to our discrete solution $E^h(t)$.
 In \cite{BellettiniKholmatov2018}, Bellettini and Kholmatov considered $A_\beta(F,E_0,\lambda)$ when $\Omega=\mathbb{R}_+^d$, the half space.
 In their case, $C_\beta(u)\geq0$ and the lower semi-continuity is easy to prove, though they invoked a flat version of the inequality (Corollary \ref{cor:d_dim_general_domain_estimate}).
 Bellettini and Kholmatov \cite{BellettiniKholmatov2018} showed the convergence of their scheme to a time evolution of sunsets (which is called a generalized minimizing movement, GMM for short).
 They proved, under conditional assumption similar to \cite{LS}, that GMM is a ``distributional'' solution of \eqref{MCFB}.

In the case $\beta\equiv0$, Chambolle \cite{Chambolle2004} proved that $E^h(t)$ actually converged to the level-set flow \cite{CGG}, \cite{ES}, see also \cite{G} of the mean curvature flow equations provided no fattening occurs;
 see e.g.\ \cite{EtoGigaIshii2012} for a generalization to general anisotropic flow for unbounded sets.
 Recently, Chambolle, Gennaro and Morini [arXiv:2212.05027] established a convergence result
 of a proposed energy minimizing scheme for the anisotropic mean curvature flow with 
 a forcing term and a mobility which depends on both the position and the direction of the normal vector.
 The minimizing movement constructed in their scheme turned out to converge to a distributional solution 
 \`{a} la Luckhaus-Sturzenhecker. The family of time step functions whose upper-level sets are equal to 
 the minimizing movement was shown to converge to the viscosity solution of the corresponding level-set equation.

In the case $\beta\not\equiv0$, the unique existence of the level-set flow has been already established by \cite{IS} and \cite{Ba}.
 We expect that our discrete solution converges to the level-set flow, although we do not try to prove it in this paper.

% 論文全体の構成
This paper is organized as follows. In Section \ref{sec:Preliminaries}, we prepare several notions and notations which will be used frequently
throughout the paper. Topics include basic convex analysis, functions of bounded variation, and geodesic distance. 
In Section \ref{sec:lsc_capillary_functional}, we present our proof of the lower semi-continuity of $C_\beta$.
In Section \ref{sec:SubdifferentialOfCapillaryFunctional},
we recall a method to characterize the subdifferential of a functional proposed by Alter. 
This method will give a concrete form of the subdifferential of the capillary functional (see Theorem \ref{Main1}).
In Section \ref{sec:CapillaryTypeChambolleScheme}, we will prove that the capillary Chambolle type scheme implements
the capillary Almgren-Taylor-Wang type scheme in some sense (see Theorem \ref{Main2}).
After that, we shall carry out a numerical experiment to confirm that our scheme works well and its outcome is as expected 
in Section \ref{sec:NumericalExperiments}. The employed scheme is the Split Bregman method. Note that we cannot apply the method directly without any modification due to contribution of the boundary energy.
No convergence result is given in this paper.

\section{Preliminaries}\label{sec:Preliminaries}
In this section, we recall several basic notions and notations without proofs which are important to 
investigate properties of the capillary functional and the capillary Chambolle type scheme.

\subsection{Convex analysis}
Let $E$ be a normed (real vector) space and $E^*$ be its conjugate (dual) space, that is the set of all bounded linear functionals on $E$. Then, for each function $f : E\to\extR$ and $u\in E$, the subdifferential $\partial f(u)$ of $f$ at $u$ is defined as the set of all elements $p$ in $E^*$ such that
\begin{equation*}
    f(v) \geq \left<p, v - u\right> + f(u)
\end{equation*}
holds for every $v\in E$, where $\left<\ ,\, \right>$ denotes the duality pair.
 Note that $f$ is allowed to take the value as $\infty$. 
If $f\not\equiv\infty$, then $f$ is called proper. The domain $D(f)\subset E$ of $f$ is defined the set of all elements $u$ in $E$ for which $f(u)<\infty$. 
Given a function $f$ on $E$, we define another function $f^*$ on $E^*$ by
\begin{equation}\label{eq:FenchelConjugate}
    f^*(p) := \sup_{u\in E}\{\left<p,u\right> - f(u)\}\tag{FC}
\end{equation}
for each $p\in E^*$. The function $f^*$ is called Fenchel conjugate of $f$. 
Let us recall one important characterization of the subdifferential in terms of Fenchel conjugate as follows; see e.g.\ \cite{Rock}.
\begin{prop}[Fenchel identity]\label{prop:FenchelIdentity}
    Assume that $f: E\to\extR$ is proper and $u\in D(f)$. Then, $p \in\subdiff{f}{u}$ if and only if the following identity is valid:
    \begin{equation*}
        f(u) + f^*(p) = \left<p, u\right>.
    \end{equation*}
\end{prop}
Fenchel identity yields another characterization of the subdifferential when $f$ is positively homogeneous of degree $1$, i.e., $f$ satisfies
\[
	f(\lambda u) = \lambda f(u)
	\ \text{for all}\ \lambda >0, \quad u \in E.
\]
\begin{prop}\label{prop:subdiff_chara}
    Assume that $f:E\to\extR$ is proper and $u\in D(f)$.
    Suppose that $f$ is positively homogeneous of degree $1$. Then, it holds that
    \begin{equation*}
        \subdiff{f}{u} = \{p\in\subdiff{f}{0}\mid f(u) = \left<p, u\right>\}.
    \end{equation*}
\end{prop}

\begin{prop}[\cite{Brezis2011}, Proposition 1.10] \label{prop:HahnBanach} % タイトルでは [ ] の中に Proposition 1.10 が入らない
    Suppose that $f:E\to (-\infty,\infty]$ is 
    lower semi-continuous and convex with $\varphi\not\equiv\infty$. 
    Then, $f$ is bounded from below by an affine continuous function.
 In other words, there exist $p\in E^*$ and $b\in\mathbb{R}$ such that
\[
	f(u) \geq \langle p,u \rangle + b
	\ \text{for all}\ u \in E.
\]
\end{prop}

\begin{rem}\label{rem:BiasCanbeZero}
The first statement of Proposition \ref{prop:HahnBanach} is a Corollary of the Hahn-Banach theorem.
 If $f$ is positively homogeneous of degree $1$, $f(\lambda v)\geq\langle p,\lambda v \rangle + b$ implies $f(v)\geq\langle p,v \rangle + b/\lambda$ for all $\lambda>0$.
 Sending $\lambda\to\infty$, we observe that
\[
	f(v) \geq \langle p,v \rangle 
	\ \text{for}\ v \in E.
\]
Thus, in the case that $f$ is positively homogeneous of degree $1$, $b$ can be taken as zero.
\end{rem}

\subsection{Function of bounded variation}
Let $\Omega$ be a smooth, bounded, and connected domain. 
Then, the total variation of $f:\Omega\to\mathbb{R}$ is defined by:
\begin{equation*}
    \int_\Omega|\nabla u| := \sup\left\{-\int_\Omega u\operatorname{div}\varphi\,dx \biggm| \varphi\in C^1_0(\Omega;\mathbb{R}^d),\ \|\varphi\|_{\infty}\leq 1.\right\}.
\end{equation*}
If $\int_\Omega|\nabla u| < \infty$, then $u$ is called a function of bounded variation in $\Omega$.
In other words, the weak derivative of $u$ is a Radon measure in $\Omega$.
Now, we let
\[
\mathbf{X}_2(\Omega) := \left\{u\in L^2(\Omega;\mathbb{R}^d) \bigm| \operatorname{div}{u}\in L^2(\Omega)\right\}.
\]
For every $z\in\mathbf{X}_2(\Omega)$ and $u\in BV(\Omega)$, the Radon measure $(z,Du)$ is defined by:

\begin{equation*}
    \left<(z,Du),\varphi\right> := -\int_\Omega u\varphi\operatorname{div}{z}\,dx - \int_\Omega uz\cdot\nabla\varphi\,dx\ \ \mbox{for}\ \ \varphi\in C^\infty_0(\Omega).
\end{equation*}
$(z,Du)$ is often called the Anzellotti pair (see \cite[Definition 1.4]{Anzellotti1983}).
Moreover, there exists a linear operator $[\cdot,\nu_\Omega]:\mathbf{X}_2(\Omega)\to L^\infty(\pOmega)$
such that $\|[z,\nu_\Omega]\|_\infty \leq \|z\|_\infty$ for each $z\in\mathbf{X}_2(\Omega)$ and
$[z,\nu_\Omega] = z\cdot\nu_\Omega$ if $z\in C^1(\closure{\Omega};\mathbb{R}^d)$ (see \cite[Theorem 1.2]{Anzellotti1983}).
The following Green's formula related to $(z,Du)$ and $[z\cdot\nu_\Omega]$ is important for our study:

\begin{equation*}
    \int_\Omega u\operatorname{div}{\varphi}\,dx = \int_{\pOmega}\gamma u[z\cdot\nu_\Omega]\,\dH{d-1} - \int_\Omega (z,Du)\ \ \mbox{for}\ \ (z,u)\in\mathbf{X}_2(\Omega)\times BV(\Omega).
\end{equation*}

\subsection{Geodesic distance}
In this section, we always assume that a domain $\Omega\subset\mathbb{R}^d$ is smooth, bounded and
connected. 
\begin{dfn}[Path]
    Let $x,y\in\closure{\Omega}$ be distinct points. Then, a Lipschitz continuous function
    $l:[0,1]\to\closure{\Omega}$ is called a \emph{path} between $x$ and $y$ if and only if $l(0) = x$ and $l(1) = y$.
\end{dfn}
\begin{dfn}[Geodesic distance between two points]
    Let $x,y\in\closure{\Omega}$ be distinct points. Then, the \emph{geodesic distance} $\operatorname{dist}_\Omega(x,y)$
    between $x$ and $y$ is defined by:
    \begin{equation}\label{eq:PPGD}
        \operatorname{dist}_\Omega(x,y) := \inf\left\{\int_0^1|l'(t)|\,dt\ \mid\ \mbox{l is a path between}\ x\ \mbox{and}\ y\right\}.\tag{PPGD}
    \end{equation}
\end{dfn}
\begin{rem}
    The infimum of \eqref{eq:PPGD} can be attained, namely a minimizer exists.
    This fact is shown in terms of the Ascoli-Arzer\'a theorem and the lower semi-continuity of $l\mapsto\int_0^1|l'(t)|dt$.
    See Section Minimal geodesics in \cite{ChenMirebeauShuCohen2019}.
\end{rem}
\begin{dfn}[Geodesic distance function]
    Let $x\in\closure{\Omega}$ and $E\subset\Omega$. Then, the geodesic distance $\operatorname{dist}_{\Omega,E}(x)$ of $x$ to $E$ is defined by:
    \begin{equation}\label{eq:PSGD}
        \operatorname{dist}_{\Omega,E}(x) := \inf\{\operatorname{dist}_{\Omega}(x,y)\ \mid\ y\in E\}.\tag{PSGD}
    \end{equation}
\end{dfn}
\begin{dfn}[Geodesic signed distance function]
    Let $E\subset\Omega$. Then, the geodesic signed distance function to $E$ is defined by:
    \begin{equation*}
        d_{\Omega,E}(x) :=
        \begin{cases}
            -\operatorname{dist}_{\Omega,\Omega\backslash E}(x)\ \ \mbox{if}\ \ x\in E,\\
            \operatorname{dist}_{\Omega,E}(x)\ \ \mbox{if}\ \ x\in \Omega\backslash E.
        \end{cases}
    \end{equation*}
\end{dfn}

If $\Omega$ is convex, then $d_{\Omega,E}$ corresponds to the ordinary signed distance function $d_{E}$ defined by
    \begin{equation*}
        d_E(x) =
        \begin{cases}
            -\inf_{y \in \Omega\backslash E} |x - y|\ \ \mbox{if}\ \ x\in E,\\
            \inf_{y \in E} |x - y|\ \ \mbox{if}\ \ x\in \Omega\backslash E.
        \end{cases}
    \end{equation*}
It is easy to see that $F\subset E$ does not necessarily imply that $d_E\leq d_F$ unless $\Omega$ is convex.
 In other words, $d_E$ is not monotonous with respect to $E$.
 This is a reason we introduce $d_{\Omega,E}$.
 Indeed, by definition we have
\begin{lem} \label{Lmono}
The geodesic signed distance is monotonous with respect to $E$.
 In other words, $F\subset E$ implies $d_{\Omega,E}\leq d_{\Omega,F}$ in $\Omega$.
\end{lem}

\section{Lower semi-continuity of capillary functional}\label{sec:lsc_capillary_functional}
In study of the energy \eqref{eq:E_b_def}, the lower semi-continuity of $C_\beta$ is crucial.
Due to the boundary integral term, this property is not straightforward.
 Nevertheless, it was already shown by Modica \cite{Modica1987} in the case where 
$\Omega$ is a $C^2$ bounded domain in $\mathbb{R}^d$ and $C_\beta$ is of the form:
\begin{equation}\label{eq:modicaCapillary}
    C_\beta(u) := \int_\Omega|\nabla u| + \int_{\pOmega}\tau(x,\gamma u(x))\,\dH{d-1}(x)\tag{ModicaCF}
\end{equation}
where $\tau:\pOmega\times\mathbb{R}\to\mathbb{R}$ is a Borel function which is $1$-Lipschitz continuous
with respect to the second variable. Note that \eqref{eq:modicaCapillary} includes 
\eqref{eq:Capillary_function} as a special case
(set $\tau(x,s) := \beta(x)s$ for each $(x,s)\in\pOmega\times\mathbb{R}$ and $\tau$ turns out to be $\|\beta\|_\infty$-Lipschitz continuous).

For the proof, he invoked a trace inequality for BV functions derived by \cite{AG}.
The inequality is of the form
\begin{equation} \label{ETrace}
	\int_{\partial\Omega}|f-g|\, d\mathcal{H}^{d-1}
	\leq \int_{\Gamma_t}\left|\nabla(f-g)\right|+\left(\frac2t + c\right)
	\int_{\Gamma_t}|f-g|\, dx
\end{equation}
for $f,g\in BV(\Omega)$ with $c$ independent of $f,g$ and $t>0$, where
\[
 \Gamma_t = \left\{ x\in\Omega \bigm| d(x)<t \right\}, \quad
 d(x) = \operatorname{dist}(x,\partial\Omega) = \inf_{y\in\partial\Omega}|x-y|.
\]
This yields
\[
	\int_{\partial\Omega}|f-g|\, d\mathcal{H}^{d-1}
	\leq \int_{\Gamma_t}|\nabla f|+\int_{\Gamma_t}|\nabla g| +\left(\frac2t + c\right)
	\int_{\Gamma_t}|f-g|\, dx.
\]
It turns out that this is enough to prove the lower semi-continuity of $C_\beta$.

In \cite{BellettiniKholmatov2018}, the lower semi-continuity of $C_\beta(u)$ is proved when $\Omega$ is the half space by using an inequality
\begin{equation} \label{Etrace2}
	\int_{\partial\Omega}|f-g|\, d\mathcal{H}^{d-1}
	\leq \int_{\Gamma_t}|\nabla f|+\int_{\Gamma_t}|\nabla g| +
	\int_{\Gamma_t}|f-g|\, dx
\end{equation}
when $\Omega$ is the half space $\mathbb{R}_+^d=\mathbb{R}^{d-1}\times(0,\infty)$.
 In \cite{BellettiniKholmatov2018}, neither the paper \cite{AG} nor \cite{Modica1987} was not mentioned.
 This type of the inequality \eqref{Etrace2} is found in \cite[Proof of Proposition 2.6, (2.11)]{Giu84}, where $\partial\Omega=B_R$ and $\Gamma_t=B_R\times(0,t)$.
 In \cite{Giu84}, it is used that trace is continuous with respect to the strict convergence in BV.
 In this paper, we establish a curved version of \eqref{Etrace2} and prove the lower semi-continuous of $C_\beta(u)$ defined by \eqref{eq:Capillary_function}.
 It works even for unbounded domains provided that it is uniformly $C^2$;
 for the definition, see e.g.\ \cite{BG}.

\begin{lem}\label{lem:d_dim_general_domain_estimate}
    Let $\Omega$ be a uniformly $C^2$ domain in $\mathbb{R}^d$ and let $\kappa_1,\cdots,\kappa_{d-1}$ 
    be the (inward) principal curvatures of $\pOmega$.
 Let $R_0$ be its reach, i.e., the largest number such that the projection $\pi:\Gamma_t\to\partial\Omega$ is well-defined for $t<R_0$, where $\left|x-\pi(x)\right|=d(x)$.
	Then, for every $f,g\in BV(\Omega)$, it holds that
    \begin{multline*}
        \int_{\pOmega}|f - g|\dH{d-1}\leq \\ \int_{\Gamma_t}\prod_{i=1}^{d-1}\frac{1}{1-(\kappa_i\circ\pi) d}|\nabla_\nu f| + \int_{\Gamma_t}\prod_{i=1}^{d-1}\frac{1}{1-(\kappa_i\circ\pi) d}|\nabla_\nu g| + \frac{1}{t}\int_{\Gamma_t}\prod_{i=1}^{d-1}\frac{1}{1-(\kappa_i\circ\pi) d}|f - g|\,dx
    \end{multline*}
    for $t\in(0,R_0)$.
    Here, $\nabla_\nu=\nabla d\cdot\nabla$ denotes the directional derivative in the direction of $\nabla d$ so that $|\nabla_\nu f|$ is well-defined as a Radon measure.
\end{lem}
\begin{proof}
Since $\pOmega$ is uniformly $C^2$, the reach $R_0$ can be taken positive.
 We take $t\in(0,R_0)$ to see that the normal coordinate system is available in $\Gamma_t$. Precisely, for each $x\in\Gamma_t$, there exists a unique $y\in\pOmega$ such that $x = y + \nu(y)d(x)$, where $\nu(y)$ denotes the inward unit normal vector to $\pOmega$ at $y$. 
    Then, we are able to use a $C^1$ change of variables between $\Gamma_t$ and $\pOmega\times(-t,0)$ defined by $\Phi:\Gamma_t\ni x\mapsto (\pi(x),d(x))\in\pOmega\times(-t,0)$; see e.g.\ \cite[Chapter 14, Appendix]{GilbargTrudinger1983}.
    For each $y\in\pOmega$, set $f_t(y) := \frac{1}{t}\int_0^tf(y + s\nu(y))ds$. Then, we compute
\begin{align*}
        &\int_{\pOmega}|f(y) - f_t(y)|\dH{d-1} = \int_{\pOmega}\left|f(y) - \frac{1}{t}\int_0^tf(y+s\nu(y))\,ds\right|\dH{d-1}(y) \\
        &= \int_0^t\frac{1}{t}ds\int_{\pOmega}\dH{d-1}\left|f(y) - f(y+s\nu(y))\right| = \int_0^t\frac{1}{t}ds\int_{\pOmega}\dH{d-1}\left|\int_0^s\deriv{u}f(y+u\nu(y))\,du\right| \\
        &= \int_0^t\frac{1}{t}ds\int_{\pOmega}\dH{d-1}(y)\left|\int_0^s\nu(y)\cdot\nabla f(y+u\nu(y))du\right| \\
        &\leq \int_0^t\frac{1}{t}ds\int_{\pOmega}\dH{d-1}(y)\int_0^s\,du\left|\nabla_\nu f(y+u\nu(y))\right| \\
        &= \int_0^t\frac{1}{t}du\int_{\pOmega}\dH{d-1}(y)\int_u^t\,ds\left|\nabla_\nu f(y+u\nu(y))\right| \leq \int_0^t\frac{1}{t}\cdot tdu\int_{\pOmega}\dH{d-1}(y)|\nabla_\nu f(y+u\nu(y))|\\
        &= \int_0^tdu\int_{\pOmega}\dH{d-1}\frac{1}{J(\Phi(y,u))}\cdot J(\Phi(y,u))|\nabla_\nu f(y+u\nu(y))| = \int_{\Gamma_t}\frac{1}{J(\Phi(x))}|\nabla_\nu f|,
    \end{align*}
    where $J$ denotes the Jacobian of $\Phi$.
    For the difference $f_t-g_t$, we obtain
    \begin{align*}
        &\int_{\pOmega}|f_t(y) - g_t(y)|\dH{d-1}(y) = \int_{\pOmega}\dH{d-1}(y)\left|\frac{1}{t}\int_0^tf(y+u\nu(y))du - \frac{1}{t}\int_0^tg(y+u\nu(y))du\right| \\
        &\leq \int_0^t\frac{1}{t}du\int_{\pOmega}\dH{d-1}\left|f(y+u\nu(y)) - g(y+u\nu(y))\right|  \\
        &= \frac{1}{t}\int_0^tdu\int_{\pOmega}\dH{d-1}(y)\frac{1}{J(\Phi(y,u))}\cdot J(\Phi(y,u))|f(y+u\nu(y)) - g(y+u\nu(y))| \\
        &= \frac{1}{t}\int_{\Gamma_t}\frac{1}{J(\Phi(x))}|f  - g| \,dx.
    \end{align*}
    Recall the exact form of the Jacobian $J(\Phi(x))$ (see e.g.\ \cite[Chapter 14, Appendix]{GilbargTrudinger1983}):
    \begin{equation*}
        J(\Phi(x)) = \prod_{i=1}^{d-1}(1 - \kappa_i(\pi(x))d(x))\ \ \mbox{for}\ \ x\in\Gamma_t.
    \end{equation*}
    Therefore,  the desired inequality follows by the triangle inequality.
\end{proof}

\begin{cor}\label{cor:d_dim_general_domain_estimate}
    Let $\Omega$ be a uniformly $C^2$ domain in $\mathbb{R}^d$ and $\mu\in(0,\infty]$ be the supremum of radius of inscribed circles of $\pOmega$. Then,
    \begin{equation*}
        \int_{\pOmega}|f - g|\,\dH{d-1}\leq \left(\frac{\mu}{\mu - t}\right)^{d-1}\int_{\Gamma_t}|\nabla f| + \left(\frac{\mu}{\mu - t}\right)^{d-1}\int_{\Gamma_t}|\nabla g| + \frac{1}{t}\left(\frac{\mu}{\mu - t}\right)^{d-1}\int_\Omega|f - g|\,dx
    \end{equation*}
    holds for every $t\in(0,\mu)$, $f,g\in BV(\Omega)$.
 In the case $\mu=\infty$, $\mu/(\mu-t)$ should be interpreted as $1$, irrelevant to $t<\infty$.
\end{cor}
\begin{proof}
    By the selection of $\mu$, it follows that $\kappa_i(\pi(x))\leq 1/\mu$ for all $1\leq i\leq d-1$ and $x\in\Gamma_t$. Moreover, we have $d(x) < t$. Thus, we can estimate as follows: 
    \begin{equation*}
        \prod_{i=1}^{d-1}\frac{1}{1-\kappa_i(\pi(x))d(x)}\leq \prod_{i=1}^{d-1}\frac{1}{1-\frac{1}{\mu}\cdot t} = \left(\frac{\mu}{\mu-t}\right)^{d-1}.
    \end{equation*}
    Hence, the desired inequality is immediately derived from Lemma \ref{lem:d_dim_general_domain_estimate} since $|\nabla_\nu f|\leq|\nabla f|$ by the Schwarz inequality and $|\nabla d|=1$.
\end{proof}

\begin{prop}\label{prop:lsc_d_dim_capillary}
    Let $\Omega$ be a uniformly $C^2$ domain in $\mathbb{R}^d$. Then, $C_\beta$ is lower semi-continuous with respect to $L^1(\Omega)$ whenever $\|\beta\|_\infty \leq 1$.
    (if $\Omega$ has a finite measure, $L^1$ can be replaced by $L^p$ for any $p \in [1,\infty)$ since $L^p(\Omega)\subset L^1(\Omega)$.)
\end{prop}
\begin{proof}
    First, we prove the assertion when $\|\beta\|_\infty < 1$.
    Let $\{u_i\}_i$ be a sequence in $L^1(\Omega)$ and $u\in L^1(\Omega)$. 
    Suppose that $u_i\to u$ in $L^1(\Omega)$ as $i\to\infty$.
    Then, it suffices to prove that
    $\limsup_{i\to\infty}\{C_\beta(u) - C_\beta(u_i)\}\leq 0$. For simplicity, set $\delta_{\mu,t} := \|\beta\|_\infty(\mu/(\mu-t))^{d-1}$.
    Then, we can take $t > 0$ so small that $\delta_{\mu,t} < 1$.
    For such a $t>0$ and for each $i\in\mathbb{N}$, we compute
    \begin{multline*}
        C_\beta(u) - C_\beta(u_i) = \int_\Omega|\nabla u| - \int_\Omega|\nabla u_i| + \int_{\pOmega}\beta(\gamma u - \gamma u_i)\,\dH{d-1}\\
        \leq \int_\Omega |\nabla u| - \int_\Omega|\nabla u_i| + \delta_{\mu,t}\int_{\Gamma_t}|\nabla u| + \delta_{\mu,t}\int_{\Gamma_t}|\nabla u_i| +\frac{1}{t}\delta_{\mu,t}\int_\Omega|u - u_i|\\
        \leq \int_\Omega |\nabla u| - \int_{\Omega\backslash\Gamma_t}|\nabla u_i| + \int_{\Gamma_t}|\nabla u| + \frac{1}{t}\int_{\Omega}|u - u_i| \\
        \leq \int_\Omega|\nabla u| - \inf_{j\geq i}\int_{\Omega\backslash\Gamma_t}|\nabla u_j| + \int_{\Gamma_t}|\nabla u_i| + \frac{1}{t}\int_\Omega|u-u_i|.
    \end{multline*}
    Letting $i\to\infty$ gives
    \begin{equation*}
        \limsup_{i\to\infty}\{C_\beta(u)-C_\beta(u_i)\} \leq \int_\Omega|\nabla u| - \liminf_{i\to\infty}\int_{\Omega\backslash\Gamma_t}|\nabla u_i| + \int_{\Gamma_t}|\nabla u|\leq 2\int_{\Gamma_t}|\nabla u|.
    \end{equation*}

    Here, we have used the lower semi-continuity of $u\mapsto\int_\Omega|\nabla u|$ with respect to $L^1(\Omega)$ - topology to obtain the last inequality.
    Since $\chi_{\Gamma_t}$ converges to $0$ as $t\to 0$ pointwise and $u\in BV(\Omega)$, 
    the Lebesgue convergence theorem gives the desired inequality by letting $t\to 0$.\newline

   We next treat the case where $\|\beta\|_\infty = 1$. We cannot apply Proposition \ref{prop:HahnBanach} directly for $C_\beta(u)$ since we do not know if it is lower semi-continuous.
   We know $C_{\beta/2}$ is lower semi-continuous by the above argument.
   By Proposition \ref{prop:HahnBanach} and Remark \ref{rem:BiasCanbeZero}, there is $g\in L^\infty(\Omega)$ such that $C_{\beta/2}(u) + \int_\Omega gu\,dx$ is non-negative for all $u\in L^1(\Omega)$.
   We set
\[
	C_{\beta,f}(u) := C_\beta(u) + \int_\Omega fu\,dx
\]
and observe that $C_{\beta,f}(u)\geq0$ for all $u\in L^1(\Omega)$ if we take $f=2g\in L^\infty(\Omega)$.
    It suffices to prove that $C_{\beta,f}$ is lower semi-continuous because $u\mapsto\int_\Omega fudx$ is continuous.
    Let us argue by contradiction. Suppose that $C_{\beta,f}$ is not lower semi-continuous.
    Then, there exist a $\delta >0$, a $u\in L^1(\Omega)$, and a sequence $\{u_i\}_i\subset L^1(\Omega)$ with $u_i\to u$ in $L^1(\Omega)$ such that $C_{\beta,f}(u) - \delta > C_{\beta,f}(u_i)$ for every $i\in\mathbb{N}$.
    We can choose $\lambda_0\in(0,1)$ such that $C_{\lambda\beta,\lambda f}(u) - \delta / 2 > C_{\beta,f}(u_i)$ for all $i\in\mathbb{N}$ whenever $\lambda_0\leq \lambda<1$.
    Since we already know that $C_{\lambda\beta,\lambda f}$ is lower semi-continuous,
    we have $\liminf_{j\to\infty} C_{\lambda\beta,\lambda f}(u_j) - \delta / 2 > C_{\beta,f}(u_i)$ (for all $i$).
    For this $\delta$, there exists a large $k\in\mathbb{N}$ such that $\inf_{j\geq k}C_{\lambda\beta, \lambda f}(u_j) - \delta / 4 > C_{\beta,f}(u_i)$.
    Then, we deduce
    
    \begin{multline}\label{eq:lsc_b_eq_1}
        \inf_{j\geq k}C_{\lambda\beta,\lambda f}(u_j) - \frac{\delta}{4} > C_{\beta,f}(u_i) \geq \lambda C_{\beta,f}(u_i) = (\lambda - 1)\int_\Omega|\nabla u_i| + C_{\lambda\beta,\lambda f}(u_i) \\\geq (\lambda-1)\int_\Omega |\nabla u_i| + \inf_{j\geq k}C_{\lambda\beta,\lambda f}(u_j).
    \end{multline}

    Here, the second inequality is derived by $C_{\beta,f}\geq 0$.
    The estimate \eqref{eq:lsc_b_eq_1} leads
    \begin{equation*}
        \frac{\delta}{4} < (1-\lambda)\int_\Omega|\nabla u_i|.
    \end{equation*}
    Since the above estimate is valid for all $\lambda\in[\lambda_0,1)$, letting $\lambda\nearrow 1$ yields a contradiction.
    Therefore, we conclude that $C_{\beta,f}$ is lower semi-continuous. The proof is complete.
\end{proof}

\section{Subdifferential of capillary functional}\label{sec:SubdifferentialOfCapillaryFunctional}
The unique solution $w^h_E$ to the minimizing problem $\operatornamewithlimits{min}_{u\in L^2(\Omega)\cap BV(\Omega)}{E^h(u)}$ 
has been comuputed as $w^h_E = d_E - \pi_{hK}d_E$ where $K = \partial C_0(0)$ and $\pi_{hK}$ 
denotes the orthogonal projection of $L^2(\Omega)$ onto $hK$ (see the discussion in \cite{Chambolle2004_TV}). 
Because of this formula, it will be useful to characterize the set $K_\beta := \partial C_\beta(0)$ 
for capturing the behaviour of the solution to the minimizing problem $\operatornamewithlimits{min}_{u\in L^2(\Omega)\cap BV(\Omega)}{\chambolleEnergy{\beta}{h}(u)}$.
To this end, let us recall an approach to the characterization due to an unpublished work by Alter explained in the book by Caselles et al. \cite{Mazon}.\newline

Let $H$ be a Hilbert space and $\Phi:H\to [0,\infty]$. 
For this $\Phi$, one can define a function $\widetilde{\Phi}:H\to [0,\infty]$ by
\begin{equation}
    \widetilde{\Phi}(u) := \sup_{v\in H}{\frac{\left<u,v\right>}{\Phi(v)}}\ \ \mbox{for}\ \ u\in H.
\end{equation}

% check!!!
\begin{rem}
    Suppose that $D(\partial\Phi)\neq\emptyset$ and $\Phi$ is positively homogeneous of degree $1$. Then, it is easy to see that 
    $\Phi^*(p) = I_K(p)$ where $K = \partial\Phi(0)$ and 
    \begin{equation*}
        I_K(p) := \begin{cases}
            0\ \ \mbox{if}\ \ p\in K,\\
            \infty\ \ \mbox{otherwise}.
        \end{cases}
    \end{equation*} 
Here, $\Phi^*$ denotes the Fenchel conjugate defined by \eqref{eq:FenchelConjugate}.
 The function $\widetilde{\Phi}$ is the support function of $\left\{v\bigm|\Phi(v)\leq1\right\}$ and it is positively homogeneous of degree $1$.
 The set $\left\{p\bigm|\widetilde{\Phi}(p)\leq1\right\}$ equals to $K$ which equals to the set $\left\{p\bigm|\Phi^*(p)<\infty\right\}$.
\end{rem}

\begin{prop}[\cite{Mazon}, Lemma 1.5]\label{prop:biconjugateOrder} % タイトルでは [ ] の中に Lemma 1.5 が入らない
    Let $\Phi_1,\Phi_2 : H\to[0,\infty]$. If $\Phi_1\leq\Phi_2$, then it holds that $\wt{\Phi_1}\geq\wt{\Phi_2}$.
\end{prop}

\begin{prop}[\cite{Mazon}, Proposition 1.6] \label{PAdual} % タイトルでは [ ] の中に Proposition 1.6 が入らない
    If $\Phi$ is convex, lower semi-continuous and positively homogeneous of degree $1$, then $\widetilde{\widetilde{\Phi}} = \Phi$.
\end{prop}

\begin{prop}[\cite{Mazon}, Theorem 1.8] \label{prop:char_subdiff} % タイトルでは [ ] の中に Theorem 1.8 が入らない
    If $\Phi$ is convex, lower semi-continuous and positively homogeneous of degree $1$, then $p\in\partial\Phi(u)$ if and only if $\widetilde{\Phi}(p)\leq 1$ and $\Phi(u) = \left<p, u\right>$.
\end{prop}

Along the line with the discussion of \cite[p.\ 15]{Mazon}, 
we are able to characterize the subdifferential $\partial C_\beta$ with the setting $H := L^2(\Omega)$ and $\Phi := C_\beta$. However, our definition of $C_\beta$ may allow itself to take a negative value. 
If so, we cannot directly apply Alter's method.
 We add a linear functional to reduce the problem for positive functionals.
 This is possible by Proposition \ref{prop:HahnBanach}.
%
%\begin{prop}(Proposition 1.10 \cite{Brezis2011})\label{prop:HahnBanach}
%    Let $E$ be a normed space and $\varphi:E\to (-\infty,\infty]$ is lower semi-continuous and convex with $\varphi\not\equiv\infty$. Then, $\varphi$ is bounded from below by an affine continuous function.
%\end{prop}
%
\begin{rem}
    When $\Omega$ is the half space $\mathbb{R}^d_+ := \mathbb{R}^{d-1}\times (0,\infty)$, 
    Bellettini and Kholmatov proved non-negativity of $\mathcal{C}_\beta$ in \cite{BellettiniKholmatov2018}.
\end{rem}

We are now in the position to prove Theorem \ref{Main2}, which characterizes the subdifferential $\partial C_\beta(u)$.
\begin{proof}[Proof of Theorem \ref{Main2}]
Since we know by Proposition \ref{prop:HahnBanach} that $C_\beta(u)$ is convex, lower semi-continuous and positively homogeneous of degree $1$ in $L^2(\Omega)$, there is $f\in L^2(\Omega)$ such that $C_{\beta,f}(u)=C_\beta(u)+\int_\Omega fu\,dx\geq0$ for all $u\in L^2(\Omega)$ by Proposition \ref{prop:HahnBanach} and Remark \ref{rem:BiasCanbeZero}.
 To prove Theorem \ref{Main2}, we introduce a functional
\begin{multline*}
	\Psi_{\beta,f}(q) = \inf \bigl\{ \|z\|_\infty \bigm|
	q = -\operatorname{div}z + f \|z\|_\infty \ \text{in}\ \mathcal{D}'(\Omega),\ 
	\|z\|_\infty \beta = -[z \cdot \nu]\ \text{on}\ \partial\Omega\ \\
	\text{with}\ z \in \mathbf{X}_2(\Omega) \bigr\}.
\end{multline*}
A key step is to prove that $\widetilde{C_{\beta,f}}=\Psi_{\beta,f}$ which is rigorously stated as follows.
\begin{lem} \label{LDual}
Let $\Omega$ be a bounded $C^2$ domain in $\mathbb{R}^d$.
 Assume that $\|\beta\|_\infty\leq1$.
 Let $f\in L^2(\Omega)$ be taken such that $C_{\beta,f}\geq0$ in $L^2(\Omega)$.
 Then $\widetilde{C_{\beta,f}}=\Psi_{\beta,f}$.
\end{lem}
\begin{rem} \label{RChar}
It is not clear that the infimum in the definition of $\Psi_{\beta,f}$ is attained.
 This causes extra technical difficulty compared with the case $\beta=0$.
\end{rem}
We continue to prove Theorem \ref{Main2} admitting Lemma \ref{LDual}.
 Since $\widetilde{C_{\beta,f}}=\Psi_{\beta,f}$ by Lemma \ref{LDual}, Proposition \ref{prop:char_subdiff} yields
\[
	q \in \partial C_{\beta,f}(u) \Leftrightarrow \Psi_{\beta,f}(q) \leq 1 \quad\text{and}\quad
	C_{\beta,f}(u) = \int_\Omega qu\,dx.
\]
For a moment, we pretend that the infimum in the definition of $\Psi_{\beta,f}$ is attained.
 In this case, $\Psi_{\beta,f}(q)\leq 1$ is equivalent to saying that $\|z_0\|_\infty=\Psi_{\beta,f}(q)\leq1$, where $z_0$ satisfies $q=-\operatorname{div}z_0+f\|z_0\|_\infty$, $\|z_0\|_\infty\beta=-[z_0\cdot\nu]$ with $z_0\in\mathbf{X}_2(\Omega)$.
 The identity $C_{\beta,f}(u)=\int_\Omega qu\,dx$ becomes
\begin{align*}
	\int_\Omega |\nabla u| &+ \int_{\partial\Omega} \beta\gamma u\, d\mathcal{H}^{d-1}
	+ \int_\Omega fu\,dx 
	= \int_\Omega u(-\operatorname{div}z_0)\, dx 
	+ \|z_0\|_\infty \int_\Omega fu\,dx \\
	&= \int_\Omega (z_0,Du) - \int_{\Omega} [z_0\cdot\nu]\gamma u\, d\mathcal{H}^{d-1}
	+ \|z_0\|_\infty \int_\Omega fu\,dx
\end{align*}
if one uses the Anzellotti pair $(z_0,Du)$.
 This implies
\begin{equation} \label{EHOM}
	\int_\Omega |\nabla u| + \int_{\partial\Omega} \beta\gamma u\, d\mathcal{H}^{d-1}
	+ \int_\Omega fu\,dx
	= \int_\Omega (z_0,Du)
	+ \|z_0\|_\infty \left( \int_{\partial\Omega} \beta\gamma u\, d\mathcal{H}^{d-1} + \int_\Omega fu\,dx \right).
\end{equation}
By \cite[Theorem C.6]{Mazon} we know % we の前にカンマはつけない
\[
	\left| \int_\Omega(z_0,Du) \right| 
	\leq \|z_0\|_\infty \int_{\Omega} |\nabla u|,
\]
which together with \eqref{EHOM} implies
\[
	C_{\beta,f}(u) \leq \|z_0\|_\infty C_{\beta,f}(u).
\]
Thus, unless $C_{\beta,f}(u)=0$, then $\|z_0\|_\infty=1$.
 In this case, by \eqref{EHOM} we have % weの前にカンマはつけない
\[
	\int_\Omega |\nabla u| = \int_\Omega(z_0,Du).
\]

Since the infimum in the definition of $\Psi_{\beta,f}$ may not be attainable, the argument is more involved.
 Let $\{z_i\}\subset \mathbf{X}_2(\Omega)$ be a minimizing sequence of the infimum in the definition of $\Psi_{\beta,f}(q)$.
 We may assume that $z_i\rightharpoonup z$ $*$-weakly in $L^\infty(\Omega,\mathbb{R}^d)$ with some $z\in L^\infty(\Omega,\mathbb{R}^d)$.
 Since
\[
	\sup_i \left( \|z_i\|_\infty + \|\operatorname{div}z_i\|_{L^2(\Omega)} \right) < \infty,
\]
we conclude that $(z_i,Du)\rightharpoonup(z,Du)$ as measures by \cite[Theorem 4.1]{Anzellotti1983}.
 Integration by parts yields \eqref{EHOM} with $z_0=z_i$.
 As in the previous paragraph, we obtain
\begin{equation} \label{EApE}
	C_{\beta,f}(u) = \int_\Omega(z_i,Du) 
	+ \|z_i\| \left( \int_{\partial\Omega} \beta\gamma u\,d\mathcal{H}^{d-1} + \int_\Omega fu\,dx \right)
	\leq \|z_i\|_\infty C_{\beta,f}(u).
\end{equation}
If $C_{\beta,f}(u)>0$, this implies that $\|z_i\|_\infty\geq1$.
 Since $\Psi_{\beta,f}(q)\leq1$, $\|z_i\|_\infty$ must converge to $1=\Psi_{\beta,f}(q)$.
 The inequality \eqref{EApE} now implies that
\[
	\lim_{i\to\infty} \int_\Omega(z_i,Du) = \int_\Omega |\nabla u|,
\]
which yields $\int_\Omega(z,Du)=\int_\Omega|\nabla u|$.
 Since $\|\cdot\|_\infty$ is lower semi-continuous under $*$-weak convergence in $L^\infty(\Omega,\mathbb{R}^d)$, we conclude that
\[
	\|z\|_\infty \leq \varliminf_{i\to\infty} \|z_i\|_\infty = \Psi_{\beta,f}(q) = 1.
\]
We take any $\varphi\in C^\infty_0(\Omega)$. Testing $q = -\operatorname{div}{z_i} + f\|z_i\|_\infty$ by $\varphi$
and sending $i\to\infty$, we see that
\begin{multline*} 
    \int_\Omega q\varphi\,dx = \int_\Omega \varphi(-\operatorname{div}{z_i})\,dx + \|z_i\|_\infty\int_\Omega f\varphi\,dx = \int_\Omega z_i\cdot\nabla\varphi\,dx + \|z_i\|_\infty\int_\Omega f\varphi\,dx\\
    \longrightarrow\int_\Omega z\cdot\nabla\varphi\,dx + \int_\Omega f\varphi\,dx.
\end{multline*}
Here, the second equality follows from \cite[Proposition C.4]{Mazon} and the last convergence is deduced 
from $z_i\rightharpoonup z$ $*-$weakly in $L^\infty(\Omega,\mathbb{R}^d)$. 
Hence, we have $q = -\operatorname{div}{z} + f$ in $\mathcal{D}'(\Omega)$.
Meanwhile, for any $\varphi\in C^\infty(\Omega)$, we again test both $q = -\operatorname{div}{z_i} + f\|z_i\|_\infty$
and $q = -\operatorname{div}{z} + f$ by $\varphi$. Then, sending $i\to\infty$ yields 
\[
    \int_{\pOmega} \beta\gamma\varphi\,\dH{d-1} = \int_{\pOmega}\gamma\varphi(-[z\cdot\nu])\,\dH{d-1}
\]
due to $\|z_i\|_\infty\beta = -[z_i\cdot\nu]$ on $\pOmega$.
Since $\varphi$ is arbitrary, we see that $\beta = -[z\cdot\nu]$.
 The converse is easy to prove.
 We thus conclude that
\[
	q \in \partial C_{\beta,f}(u)
\]
is equivalent to saying that
\begin{gather*}
	q = -\operatorname{div}z + f, \quad
	\beta = -[z, \nu]\ \text{with}\ z \in \mathbf{X}_2(\Omega)\ \text{and}\ \|z\|_\infty \leq 1, \\
	\int_\Omega |\nabla u| = \int_\Omega(z,Du).
\end{gather*}
Since $\partial C_{\beta,f}(u)=\partial C_\beta(u)+f$, the characterization of $\partial C_\beta(u)$ in Theorem \ref{Main2} now follows provided that $C_{\beta,f}(u) > 0$.

If $C_{\beta,f}(u)=0$, then
\[
	C_{\beta,f}(u) = \int_\Omega qu\,dx
\]
implies
\[
	C_{\beta,f}(u) = \int_\Omega \alpha qu\,dx
\]
for all $\alpha\in\mathbb{R}$.
 Thus,
\[
	\partial C_{\beta,f}(u) = \left\{ \lambda q \in L^2(\Omega) \biggm|
	q \in L^2(\Omega)\ \text{and}\ \lambda \in [0,1]\ \text{with}\ \Psi_{\beta,f}(q) = 1\ \text{and}\ \int_\Omega qu\,dx = 0 \right\}.
\]
As we observed, $\Psi_{\beta,f}(q) = 1$ with $\int_\Omega qu\,dx =C_{\beta,f}(u)$ is equivalent to saying that there exists $z\in L^\infty(\Omega,\mathbb{R}^d)$ such that
\begin{gather*}
	q = -\operatorname{div}z + f, \quad
	\beta = -[z, \nu]\ \text{with}\ z \in \mathbf{X}_2(\Omega)\ \text{and}\ \|z\|_\infty \leq 1, \\
	\int_\Omega |\nabla u| = \int_\Omega(z,Du).
\end{gather*}
Thus, we obtain the desired characterization of $\partial C_\beta(u)$ in Theorem \ref{Main2}.

\end{proof}
\begin{proof}[Proof of Lemma \ref{LDual}]
The proof is similar to the case $\beta=0$, $f=0$ in \cite[Proposition 1.9]{Mazon}.
 If $\Psi_{\beta,f}(q)=\infty$, then $\widetilde{C_{\beta,f}}(q)\leq\Psi_{\beta,f}(q)$ so we may assume that $\Psi_{\beta,f}(q)<\infty$.
 Let $u\in L^2(\Omega)\cap BV(\Omega)$.
 For $q=-\operatorname{div}z+f\|z\|_\infty$, $\|z\|_\infty\beta=-[z\cdot\nu]$ with $z\in\mathbf{X}_2(\Omega)$, we observe that
\begin{align*} % = と ( を揃える
	&\int_\Omega uq\,dx &&\hspace{-6em}= \int_\Omega u \left( -\operatorname{div}z + f\|z\|_\infty \right)dx \\
	&&&\hspace{-6em}= \int_\Omega (z,Du) + \int_\Omega fu\,dx \|z\|_\infty
	+ \int_{\partial\Omega} \beta\gamma u\,d\mathcal{H}^{d-1}\|z\|_\infty \\
	&\leq \|z\|_\infty &&\hspace{-6em}\left( \int_\Omega |\nabla u| + \int_{\partial\Omega} \beta\gamma u\,d\mathcal{H}^{d-1} + \int_\Omega fu\,dx \right)
	= \|z\|_\infty C_{\beta,f}(u).
\end{align*}
Since $u\in BV(\Omega)\cap L^2(\Omega)$ is dense in $L^2(\Omega)$, taking supremum in $u$ implies that $\widetilde{C_{\beta,f}}(q)\leq\|z\|_\infty$.
The inequality $\widetilde{C_{\beta,f}}(q)\leq\Psi_{\beta,f}(q)$ now follows by taking the infimum of $\|z\|_\infty$.

For the converse inequality, it suffices to prove that  $C_{\beta,f}(u)\leq\widetilde{\Psi_{\beta,f}}(u)$ by Proposition \ref{prop:biconjugateOrder} and Proposition \ref{PAdual}.
 We may assume that $u\in L^2(\Omega)\cap BV(\Omega)$.
 We proceed
\[
	\widetilde{\Psi_{\beta,f}}(u)
	= \sup_{q\in L^2(\Omega)} \frac{\int_\Omega uq\,dx}{\Psi_{\beta,f}(q)}
	\geq \sup_{\substack{q\in L^2(\Omega)\\\Psi_{\beta,f}(q)<\infty}} \frac{\int_\Omega uq\,dx}{\Psi_{\beta,f}(q)}.
\]
If $\Psi_{\beta,f}(q)<\infty$, then we observed, for $q=-\operatorname{div}z+f\|z\|_\infty$, $\|z\|_\infty\beta=-[z\cdot\nu]$ with $z\in\mathbf{X}_2(\Omega)$,
\[
	\int_\Omega uq\,dx = \int_\Omega(z,Du)
	+ \int_\Omega fu\,dx \|z\|_\infty + \int_{\partial\Omega}\beta\gamma u\,d\mathcal{H}^{d-1} \|z\|_\infty.
\]
Thus,
\begin{align*}
	\widetilde{\Psi_{\beta,f}}(u) \geq &\sup_{q\in L^2(\Omega)} \frac{\int_\Omega uq\,dx}{\|z\|_\infty}
	\geq \sup \frac{1}{\|z\|_\infty}\int_\Omega (z,Du) \\
	& + \int_{\partial\Omega} \beta\gamma u\,d\mathcal{H}^{d-1} + \int_\Omega fu\,dx, 
\end{align*}
where the last supremum is taken for $z\in\mathbf{X}_2(\Omega)$ satisfying
\[
	\|z\|_\infty \beta = -[z\cdot\nu] \quad\text{on}\quad \partial\Omega.
\]
To conclude that $\widetilde{\Psi_{\beta,f}}(u)\geq C_{\beta,f}(u)$, it suffices to prove that
\[
	\sup \left\{ \frac{1}{\|z\|_\infty}\int_\Omega (z,Du) \biggm|
	z \in \mathbf{X}_2(\Omega),\ \|z\|_\infty \beta = -[z\cdot\nu] \ \text{on}\ \partial\Omega \right\} \geq \int_\Omega |\nabla u|.
\]
 Since $|\nabla u|$ is a Radon measure in $\Omega$, for $\varepsilon>0$ there is $\delta>0$ such that % there の前にカンマはつけない
\[
	\int_{\Omega_{\delta}} |\nabla u| \geq \int_\Omega |\nabla u| - \varepsilon
\]
for $\Omega_\delta=\left\{x\in\Omega \bigm| \operatorname{dist}(x,\partial\Omega)>\delta \right\}$.
 For any $z\in C^\infty(\Omega_{\delta/2})$ which is compactly supported in $\Omega_\delta$, we are able to extend $z$ to $\Omega$ such that the extended $z$ satisfies $z\in\mathbf{X}_2(\Omega)$ with $\|z\|_\infty\beta=-[z\cdot\nu]$ on $\partial\Omega$.
 For such $z$ % カンマはつけない
\[
	\int_{\Omega_\delta} (z,Du)
	= \int_{\Omega_\delta} v(-\operatorname{div}z)\, dx.
\]
We know that
\[
	\int_{\Omega_\delta} |\nabla u|
	= \sup_{z\in C_0^\infty(\Omega_\delta)}
	\frac{1}{\|z\|_\infty}\int_{\Omega_\delta} u(-\operatorname{div}z)\, dx.
\]
Thus,
\begin{multline*}
	\sup \left\{ \frac{1}{\|z\|_\infty}\int_\Omega (z,Du) \biggm|
	z \in \mathbf{X}_2(\Omega),\ \|z\|_\infty\beta=-[z\cdot\nu]\ \text{on}\ \partial\Omega \right\} \\
	\geq \operatornamewithlimits{sup}{\frac{1}{\|z\|_\infty}\int_{\Omega_\delta}u(-\operatorname{div}{z})\,dx}\geq\int_{\Omega_\delta} |\nabla u| \geq \int_\Omega|\nabla u| - \varepsilon,
\end{multline*}
where the second supremum is taken over $z\in C^\infty_0(\Omega)$ such that $\|z\|_\infty\beta = -[z\cdot\nu]$ on $\pOmega$.
Since $\varepsilon$ is arbitrary, we now conclude the desired inequality.
 Thus, we have proved that $\widetilde{\Psi_{\beta,f}}(u)\geq C_{\beta,f}(u)$.
\end{proof}

%%%%%%%
\section{Capillary Chambolle type scheme}\label{sec:CapillaryTypeChambolleScheme}
We will show that Chambolle's scheme is a concrete way to implement 
Almgren-Taylor-Wang's scheme.
In other words, we shall prove Theorem \ref{Main1}.
\begin{proof}[Proof of Theorem \ref{Main1}]
    Existence and uniqueness of the minimizer $w$ of $\chambolleEnergy{\beta}{h}$ follows 
    from strict convexity and the lower semi-continuity of the energy with respect to $L^2(\Omega)$. 
    The Euler-Lagrange inclusion of \eqref{eq:E_b_def} reads

    \begin{equation}
        \frac{w - \geodis{E_0}}{h} + \partial C_\beta(w) \ni 0.
    \end{equation}

    We set $p := (w - \geodis{E_0})/h$ for simplicity. 
    Then, we have $-p \in \partial C_\beta(w)$. 
    There exists an $M>0$ such that $|\geodis{E_0}| \leq M$.
    Then, we deduce from the maximal principle that $|w|\leq M$. 
    Set $F_s := \{w < s\}$ for each $s\in\mathbb{R}$. 
    Noting that $w(x) = M - \int_{w(x)}^Mds = M - \int_{-M}^M \chi_{F_s}ds$, we obtain 

    \begin{equation}\label{eq:C_b_1}
        C_\beta(w) = -\int_\Omega pw\,dx = \int_{-M}^M \int_\Omega p\chi_{F_s}\,dxds - M\int_\Omega p\,dx.
    \end{equation}

    The first equality is derived by $-p\in \partial C_\beta(w)$. 
    Since $-p\in \partial C_\beta(0)$, 
    it holds that $C_\beta(u) \geq - \int_\Omega pu dx$ for every $u\in L^2(\Omega)$. 
    In particular, substituting $1, -1\in L^2(\Omega)$ into this inequality gives

    \begin{equation*}
        \int_{\partial\Omega}\beta\,\dH{d-1} \geq -\int_\Omega p\,dx,\ \ -\int_{\partial\Omega}\beta\,\dH{d-1} \geq \int_\Omega  p\,dx
    \end{equation*}
    implying
    \begin{equation}\label{eq:integralOfp}
        \int_\Omega p\,dx = -\int_{\partial\Omega} \beta\,\dH{d-1} = 0.
    \end{equation}
    Here the average-free assumption on $\beta$ is invoked.

We next observe that
    \begin{equation}\label{eq:C_b_2}
        C_\beta(w) = \int_{-M}^M \left(\int_\Omega |\nabla\chi_{F_s}| - \int_{\partial\Omega}\beta\chi_{F_s}\dH{d-1}\right) ds.
    \end{equation}
    To get \eqref{eq:C_b_2}, we have used the co-area formula with respect to BV functions:
    \begin{equation}
        \int_\Omega |\nabla w| = \int_{-\infty}^\infty \int_\Omega |\nabla\chi_{F_s}|
    \end{equation}
    Moreover, since $\chi_{F_s}\equiv 1$ for all $s\in(M,\infty)$ and $\chi_{F_s}\equiv 0$ 
    for all $s\in(-\infty,-M)$, we deduce
    \begin{equation*}
        \int_\Omega|\nabla w| = \int_{-M}^M\int_\Omega|\nabla\chi_{F_s}|.
    \end{equation*}
For the term containing $\beta$, we have
    \begin{multline*}
        \int_{\pOmega}\beta w\,\dH{d-1} = -\int_{-\infty}^0\int_{\pOmega}\beta\chi_{\{w < s\}}\dH{d-1}ds + \int_0^\infty\int_{\pOmega}\beta\chi_{\{w > s\}}\dH{d-1}ds\\ 
        = -\int_{-M}^0\int_{\pOmega}\beta\chi_{\{w<s\}}\dH{d-1}ds + \int_0^\infty\int_{\pOmega}\beta(1-\chi_{\{w < s\}})\dH{d-1}ds \\
        = -\int_{-M}^0\int_{\pOmega}\beta\chi_{\{w<s\}}\dH{d-1}ds - \int_0^M\int_{\pOmega}\beta\chi_{\{w<s\}}\dH{d-1}ds = -\int_{-M}^M\int_{\pOmega}\beta\chi_{F_s}\dH{d-1}ds.
    \end{multline*}
Thus, the formula \eqref{eq:C_b_2} follows.

Combining \eqref{eq:C_b_1} and \eqref{eq:C_b_2} yields

    \begin{equation} \label{Edis}
        \int_{-M}^M C_{-\beta}(\chi_{F_s})\,ds = \int_{-M}^M\int_\Omega p\chi_{F_s} dxds.
    \end{equation}
Since $-p\in\subdiff{C_\beta}{w}$, we see that $p\in\subdiff{C_{-\beta}}{-w}\subset\subdiff{C_{-\beta}}{0}$.
     Thus, it follows that $C_{-\beta}(\chi_{F_s})\geq\int_\Omega p\chi_{F_s}dx$.
     Therefore, the identity \eqref{Edis} yields $C_{-\beta}(\chi_{F_s}) = \int_\Omega p\chi_{F_s}ds$ holds for a.e. $s\in[-M,M]$. 
     This can be rephrased as $p\in\partial C_{-\beta}(\chi_{F_s})$ by Proposition \ref{prop:char_subdiff} since $C_\beta$ is positively homogeneous of degree $1$.
     For such $s\in[-M,M]$, taking any $F\subset\Omega$, we obtain the estimate of the form:
     \begin{multline*}
         C_{-\beta}(\chi_F)\geq \int_\Omega p(\chi_F - \chi_{F_s})\,dx + C_{-\beta}(\chi_{F_s}) = \int_\Omega\frac{w-\geodis{E_0}}{h}\cdot(\chi_F - \chi_{F_s})dx + C_{-\beta}(\chi_{F_s}) \\
         = \int_\Omega\frac{w - s}{h}\cdot(\chi_F - \chi_{F_s})\,dx + \int_\Omega\frac{s - \geodis{E_0}}{h}\cdot(\chi_F - \chi_{F_s})\,dx + C_{-\beta}(\chi_{F_s}) \\
         \geq \int_{\Omega}\frac{s-\geodis{E_0}}{h}\chi_F\,dx - \int_{\Omega}\frac{s-\geodis{E_0}}{h}\chi_{F_s}\,dx + C_{-\beta}(\chi_{F_s}).
     \end{multline*}
     This leads
     \begin{equation*}
         C_{-\beta}(\chi_F) + \int_{\Omega\cap F}\frac{\geodis{E_0}-s}{h}\,dx \geq C_{-\beta}(\chi_{F_s}) + \int_{\Omega\cap F_s}\frac{\geodis{E_0}-s}{h}\,dx.
     \end{equation*}
     We set $E_s := \{\geodis{E_0} < s\}$ for each $s\in\mathbb{R}$. Noting that
     \begin{equation*}
         \int_{\Omega\cap(F\triangle E_s)}\frac{|\geodis{E_0} - s|}{h} = \int_{\Omega\cap F}\frac{\geodis{E_0} - s}{h}\,dx - \int_{\Omega\cap E_s}\frac{\geodis{E_0} - s}{h}\,dx,
     \end{equation*}
     $F_s$ turns out to be a minimizer of $F\mapsto C_{-\beta}(\chi_F) + \int_{\Omega\cap(F\triangle E_s)}|\geodis{E_0}-s|/h$.
     We can take a decreasing sequence $s_i\to 0$ as $i\to\infty$ such that
     $p\in\subdiff{C_{-\beta}}{\chi_{F_{s_i}}}$ for every $i\in\mathbb{N}$.
     Then, since $\chi_{F_{s_i}}\to\chi_{F_0}$ in $L^2(\Omega)$, the lower semi-continuity of $C_{-\beta}$ implies
     \begin{equation*}
         C_{-\beta}(\chi_{F_0}) \leq \liminf_{i\to\infty}C_{-\beta}(\chi_{F_{s_i}}) = \liminf_{i\to\infty}\int_\Omega p\chi_{F_{s_i}}dx = \int_\Omega p\chi_{F_0}dx.
     \end{equation*}
     Here, the last equality follows from pointwise convergence of $\chi_{F_{s_i}}$ to $\chi_{F_0}$ and the Lebesgue convergence theorem.
     The converse inequality is derived from $p\in\subdiff{C_{-\beta}}{0}$.
     Therefore, we conclude $p\in\subdiff{C_{-\beta}}{\chi_{F_0}}$ which leads $T^\beta_h(E_0) = \operatorname{argmin}_{F}\mathcal{A}_{-\beta}(F,E_0,1/h)$.
\end{proof}
In Theorem \ref{Main1}, $\beta$ cannot be constant
owing to the restriction $\int_\Omega \beta\dH{d-1} = 0$. 
Moreover, we are not sure that the given $\beta$ exactly describes the desired contact angle condition
because we do not know the position of $\partial^*E_0\cap\pOmega$.
Thus, we have to define $\beta$ as $-\cos{\theta(0,\cdot)}$ in a neighbor of the 
boundary $\partial E_0\cap\pOmega$ of the hypersurface, and we set $\beta$ as a constant
so that $\int_\Omega\beta\dH{d-1} = 0$.
Subsequently, we rigorously state how to implement Chambolle's scheme 
with capillary functional. \newline

Let $E_0\subset\Omega$ be a Caccioppoli set in $\Omega$, and suppose that $\mathcal{H}^{d-1}(\partial^*E_0\cap\pOmega) = 0$ and $\partial^*E_0\cap\partial\Omega$ is not dense in $\partial\Omega$.
Let $T > 0$ be a time horizon and $h>0$ be a time step. Then, the time interval $[0,T]$ is split into $N$ sub-intervals,
where $N\in\mathbb{N}$ and $T = hN$. Suppose that $\theta:[0,T]\times\pOmega\to[0,\pi]$ is given.
Then, we define a function $\beta_{h,0}$ as follows:
\begin{equation*}
    \beta_{h,0} := 
    \begin{cases}
        -\cos{\theta(0,\cdot)}\ \ \mbox{if}\ \ \mathcal{N}_{h,0},\\
        \frac{\int_{\mathcal{N}_{h,0}}\cos{\theta(0,\cdot)}\dH{d-1}}{\mathcal{H}^{d-1}(\pOmega) - \mathcal{H}^{d-1}(\mathcal{N}_{h,0})}\ \ \mbox{on}\ \ \pOmega\backslash\mathcal{N}_{h,0},
    \end{cases}
\end{equation*}
where $\mathcal{N}_{h,0}\subset\pOmega$ is a neighborhood of $\partial^*E_0\cap\pOmega$ such that $\partial\Omega\backslash\mathcal{N}_{h,0}$ has positive $\mathcal{H}^{d-1}$-measure on $\partial\Omega$.

Let $w_{h,0}\in L^2(\Omega)\cap BV(\Omega)$ be the unique minimizer of 
the energy $\chambolleEnergy{\beta_{h,0}}{h}$. Then, we set $T_h(E_0) := \{ w_{h,0} < 0\}$.
Next, $\beta_{h,1}$ is defined in terms of $T_h(E_0)$ as follows:

\begin{equation*}
    \beta_{h,1} := 
    \begin{cases}
        -\cos{\theta(\frac{1}{h},\cdot)}\ \ \mbox{if}\ \ \mathcal{N}_{h,1},\\
        \frac{\int_{\mathcal{N}_{h,1}}\cos{\theta(\frac{1}{h},\cdot)}\dH{d-1}}{\mathcal{H}^{d-1}(\pOmega) - \mathcal{H}^{d-1}(\mathcal{N}_{h,1})}\ \ \mbox{on}\ \ \pOmega\backslash\mathcal{N}_{h,1},
    \end{cases}
\end{equation*}
where $\mathcal{N}_{h,1}\subset\pOmega$ is a neighborhood of $\partial^*(T_h(E_0))\cap\pOmega$.
Then, we set $T_h^2(E_0) := \{ w_{h,1} < 0 \}$ where $w_{h,1}$ is the unique minimizer of $\chambolleEnergy{\beta_{h,1}}{h}$.
By the inductive step, we define $\mathcal{N}_{h,i}$, $\beta_{h,i}$ and $T^{i+1}_h(E_0)$ for $0\leq i\leq N-1$ with $T_h^1(E_0) := T_h(E_0)$ assuming that at each step $\mathcal{N}_{h,1}$ can be taken with the property that $\partial\Omega\backslash\mathcal{N}_{h,1}$ has positive $\mathcal{H}^{d-1}$-measure on $\partial\Omega$.

Under these notations, we define a time discrete evolution $E_h(t)$ of sets in $\Omega$ by:
\begin{equation*}
    E_h(t) := T_h^{\lfloor\frac{t}{h}\rfloor}(E_0)\ \ \mbox{for}\ \ t\in[0,T].
\end{equation*}

\begin{rem}
    The definition of $E_h$ depends not only on the choice of $h$, but also on $\mathcal{N}_{h,i}$.
    If one tends to prove the convergence of the proposed discrete scheme, then independence of 
    the choice of $\mathcal{N}_{h,i}$ should be shown as well.
    Though we expect that $E_h(t)$ will converge to a time evolution $E(t)$ in some sense,
    we do not provide any convergence result and leave it for future works. Alternatively,
    we shall carry out several numerical experiments to confirm that the proposed scheme works well and
    behaves as desired.
\end{rem}

\section{Numerical experiment}\label{sec:NumericalExperiments}
In this section, we show how the discrete scheme works through some examples. Our scheme consists of the following parts.
\begin{enumerate}
    \item Given an initial data $E\subset\Omega$, compute the signed distance $d_E$ in terms of fast marching algorithm.
    \item Solve the isotropic TV denoising problem with the initial data $d_E$ by the Split Bregman method. Let $w_E$ be a solution of it.
    \item Compute the zero sublevel set $E'$ of $w_E$ and set $E := E'$.
    \item Repeat the process from 1 to 3.
\end{enumerate}
Our goal in this section is to modify the scheme mentioned above so that it also works in the case where TV is replaced by the capillary functional $C_\beta$ for some $\beta\in L^\infty(\pOmega)$ and to verify its accuracy through some strong solutions to the mean curvature flow with contact angle condition.
 Nevertheless, we begin with a classical case to get in touch with the basic idea of this method. By the classical case, we mean that $E$ is sequentially compact in $\Omega$, that is $E\subset\subset\Omega$.

\subsection{How to derive distance function?}
To compute the signed distance function $d_E$ numerically, we put collocation points $\mathbf{X}_i\ (1\leq i\leq N)$ and regard $E$ as the polygon $\cup_{i=1}^N[\mathbf{X}_{i-1},\mathbf{X}_i]$ with solid. For short, this polygon will be still denoted by $E$. Then, we can judge whether each point $x$ in $\Omega$ is included in the polygon or not by investigating the winding number of the polygon around $x$. In this way, we have the discrete function $d_E:\{1,\cdots,N_y\}\times\{1,\cdots,N_x\}\to\{-1,1\}$ defined by
\begin{equation*}
    d_E(i,j) :=
    \begin{cases}
        1\ \ \mbox{if}\ \ (x_j,y_i)\in E, \\ 
        -1\ \ \mbox{if}\ \ (x_j,y_i)\notin E.
    \end{cases}
\end{equation*}
Then, we obtain the approximate values of $d_E$ at mesh points near $\partial E$ by applying the argument done in \S 3.1 in \cite{Chambolle2004}. As stated in \cite{Chambolle2004}, we are now in the position to start the fast marching algorithm to determine the approximate values of $d_E$ far from $\partial E$. To this end, we have utilized "FastMarching.jl", a library of Julia developed by Hellemo [Github; hellemo/FastMarching.jl; accessed; 2023 May 21]. FastMarching.jl accepts coordinate of mesh points nearby $\partial E$ and returns distance between each mesh point and $\partial E$. After that, we finally update the sign of each $d_E(i,j)$ by checking the sign of $w_E(i,j)$. In the second iteration, we do not have to calculate the winding number of $E$ because we already know the level set function $w_E$.

\subsection{Split Bregman method}
% fast_marching_algorithm.jl -> Trial of FMM
% split_bregman.jl -> minimize the energy by Split bregman..
% dl_chambolle_boundary.jl -> minimize the energy by Deep learning..
% dl_replay_stational_2d.jl -> showing distance function we derive by solving the eikonal equation.
% dl_replay_move_2d.jl -> showing time evolution of layers.
% dl_measure_soliton.jl -> compute difference between rigorous translating soliton and the obtained approximate solution.
% mcf_soliton.jl -> time evolution of the rigorous translating soliton.
Let us recall the Split Bregman method first proposed by Goldstein and Osher \cite{GoldsteinOsher2009}. Their scheme aims to solve problems that are categorized in the class of $L^1$ regularized optimization problem. Before applying their scheme to our problem, let us briefly review the proposed scheme. Let $f$ be a given data and $\mu >0$, they considered the following energy minimizing problem:

\begin{equation}\label{eq:energy_split_bregman}
    \min_{u}\TV{\Omega}{u} + \frac{\mu}{2}\|u - f\|^2_2.
\end{equation}
Note that this quantity is nothing but the energy to be minimized in Chambolle's scheme if one chooses $\mu := \frac{1}{h}$ and $f:=d_E$ for some given initial data $E\subset\Omega$. In our problem, this corresponds to the case where $\beta\equiv 0$. The idea of the Split Bregman method is to divide the variable $u$ of \eqref{eq:energy_split_bregman} into two portions $u$ and $\mathbf{d} = (d_x,d_y) :=\nabla u$ and to solve alternatively the following problem:
\begin{equation}\label{eq:energy_split_bregman_2}
    \min_{u,d_x,d_y}\int_\Omega|\mathbf{d}| + \frac{\mu}{2}\|u - f\|^2_2 + \frac{\lambda}{2}\|d_x - u_x\|^2_2 + \frac{\lambda}{2}\|d_y - u_y\|^2_2.
\end{equation}
The last two terms are regarded as a constraint $\mathbf{d} = \nabla u$ and \eqref{eq:energy_split_bregman_2} is an unconstrained problem.
 Note that the problem under consideration is represented as the sum of $L^1$ and $L^2$ terms. The minimizer $u$ is approximated by a sequence $\{\iter{u}{k}{}\}_{k}$ of functions that are generated an iterate step. To this end, setting $\iter{u}{0}{} := f$ and $\iter{d_x}{0}{} = \iter{d_y}{0}{} = \iter{b_x}{0}{} = \iter{b_y}{0}{} = 0$, $\iter{u}{k}{}$, $\iter{d_x}{k}{}$, $\iter{d_y}{k}{}$, $\iter{b_x}{k}{}$ and $\iter{b_y}{k}{}$ $(k\in\mathbb{N})$ are determined by the following equality:
\begin{equation}\label{eq:minimizing_d}
    \iter{\mathbf{d}}{k}{} := \argmin_{\mathbf{d}}|\mathbf{d}| + \frac{\lambda}{2}\|\mathbf{d} - \nabla \iter{u}{k-1}{} - \iter{b}{k-1}{}\|^2_2
\end{equation}
and
\begin{equation}\label{eq:minimizing_u}
    \iter{u}{k}{} := \argmin_{u}\frac{\mu}{2}\|u - f\|^2_2 + \frac{\lambda}{2}\|\iter{\mathbf{d}}{k-1}{} - \nabla u - \iter{b}{k-1}{}\|^2_2.
\end{equation}
The operator to derive $d^{(k)}$ from $\nabla u^{(k-1)}$, $b^{(k-1)}$ is often called a shrinking operator, and it can be calculated explicitly without differentiating $|d|$.
 In the concrete procedure in numerical computation, we assume that $\Omega = (\alpha_x,\alpha_y)\times(\beta_x,\beta_y)\subset\mathbb{R}^2$ for some $\alpha_x < \beta_x$ and $\alpha_y < \beta_y$ and the functions $u$, $d_x$, $d_y$, $b_x$ and $b_y$ are defined on mesh points $(x_j,y_i)$ $(1\leq j\leq N_x, 1\leq i\leq N_y)$ of $\Omega$ where $N_x$ and $N_y$ are the number of meshes along the $x$-axis and the $y$-axis, respectively and $x_j := \alpha_x + \frac{(\beta_x - \alpha_x)j}{N_x}$ and $y_i := \alpha_y + \frac{(\beta_y - \alpha_y)i}{N_y}$.
 For simplicity, we write $u_{i,j} := u(x_j, y_i)$. 
Then, the minimizers $\iter{\mathbf{d}}{k}{i,j} = (\iter{d_x}{k}{i,j},\iter{d_y}{k}{i,j})$ and $\iter{u}{k}{i,j}$ of  \eqref{eq:minimizing_d} and \eqref{eq:minimizing_u} can be explicitly computed as follows:
\begin{equation*}
    \iter{d_x}{k}{i,j} = \frac{\iter{s}{k-1}{i,j}\lambda(\nabla_x\iter{u}{k-1}{i,j} + \iter{b_x}{k-1}{i,j})}{\iter{s}{k-1}{i,j}\lambda + 1},
    \iter{d_y}{k}{i,j} = \frac{\iter{s}{k-1}{i,j}\lambda(\nabla_y\iter{u}{k-1}{i,j} + \iter{b_y}{k-1}{i,j})}{\iter{s}{k-1}{i,j}\lambda + 1},
\end{equation*}
\begin{equation}\label{eq:DiscreteEulerLagrange}
    \iter{u}{k}{i,j} = \iter{G}{k-1}{i,j},
    \iter{b_x}{k}{i,j} = \iter{b_x}{k-1}{i,j} + (\nabla_x\iter{u}{k}{i,j} - \iter{d_x}{k}{i,j}),
    \iter{b_y}{k}{i,j} = \iter{b_y}{k-1}{i,j} + (\nabla_y\iter{u}{k}{i,j} - \iter{d_y}{k}{i,j}).
\end{equation}
Here, we have set
\begin{equation*}
    \iter{s}{k-1}{i,j} := \sqrt{|\nabla_x\iter{u}{k-1}{i,j} + \iter{b_x}{k-1}{i,j}|^2 + |\nabla_y\iter{u}{k-1}{i,j} + \iter{b_y}{k-1}{i,j}|^2}, \\
\end{equation*}
and
\begin{multline}\label{eq:EulerLagrangeInBulkSplitBregmanDiscrete}
    \iter{G}{k}{i,j}  := \frac{\lambda}{\mu\Delta x^2 + 4\lambda}\left(\iter{u}{k-1}{i+1,j} + \iter{u}{k-1}{i-1,j}+\iter{u}{k-1}{i,j+1} + \iter{u}{k-1}{i,j-1}\right) \\+ \frac{\Delta x^2}{\mu\Delta x^2 + 4\lambda}(\mu f_{i,j} + \lambda(\nabla_x(\iter{d_x}{k-1}{i,j} - \iter{b_x}{k-1}{i,j}) + \nabla_y(\iter{d_y}{k-1}{i,j} - \iter{b_y}{k-1}{i,j}))))
\end{multline}
where $\Delta x:= \frac{\beta_x - \alpha_x}{N_x} = \frac{\beta_y - \alpha_y}{N_y}$ and $\iter{b}{k}{}$ is determined in the course of Bregman iteration. 
The procedure \eqref{eq:DiscreteEulerLagrange} is repeated until the following criteria holds:

\begin{equation*}
    \sqrt{\sum_{i=1}^{N_y}\sum_{j=1}^{N_x}(\iter{u}{k}{i,j} - \iter{u}{k-1}{i,j})^2\Delta x^2} =: \|\iter{u}{k}{} - \iter{u}{k-1}{}\| \geq 10^{-3}.
\end{equation*}
The partial derivatives $\nabla_x u$ and $\nabla_y u$ are discretized by the central finite difference, namely
\begin{equation*}
    \nabla_x u_{i,j} := \frac{u_{i,j+1} - u_{i,j-1}}{2\Delta x},\ \nabla_y u_{i,j} := \frac{u_{i+1,j} - u_{i-1,j}}{2\Delta x}.
\end{equation*}
We have to care when the point $(i,j)$ is on $\pOmega$, that is the case either $i=1$, $i=N_y$, $j = 1$ or $j = N_x$. In such cases, the neighbor points $(i+1,j)$, $(i-1,j)$, $(i,j+1)$ and $(i,j-1)$ may not belong to $\Omega$. Then, we shall impose either the periodic boundary condition or Neumann boundary condition to the function $u$. This depends on the selection of the initial data $f$. We will decide which boundary condition should be used in each specific problem.

For derivation of the discrete scheme, we refer the reader to Split Bregman Isotropic TV Denoising \cite{GoldsteinOsher2009}. Therein, the mesh size $\Delta x$ is assumed to be $1.0$ so that the formulae seem much simpler than that of us.

\subsection{Closed curves}
We begin with classical cases, namely the case where an initial data with solid is fully included in $\Omega$. In this case, we always impose the periodic boundary condition to minimizers $u$ to be determined through our scheme. Precisely speaking, we assume that $u_{1,j} = u_{N_y,j}\ (1\leq j\leq N_x)$, $u_{i,1} = u_{i,N_x}\ (1\leq i\leq N_y)$. Curves presented below are zero level lines of minimizers which are derived through our scheme.
\subsubsection{Star shaped curve}
In this section, the initial curve is parameterized as $(0,2\pi)\ni t\mapsto ((3.0 + \sin{5t})\cos{t},(3.0 + \sin{5t})\sin{t})\in\mathbb{R}^2$. 
The hyperparameters are set as $\alpha_x := -5.0$, $\beta_x := 5.0$, $\alpha_y := -5.0$, $\beta_y := 5.0$, $N_x := 500$ and $N_y:=500$.
\begin{figure}[H]
    \begin{minipage}[b]{0.5\linewidth}
      \centering
      \includegraphics[keepaspectratio, scale=0.4]{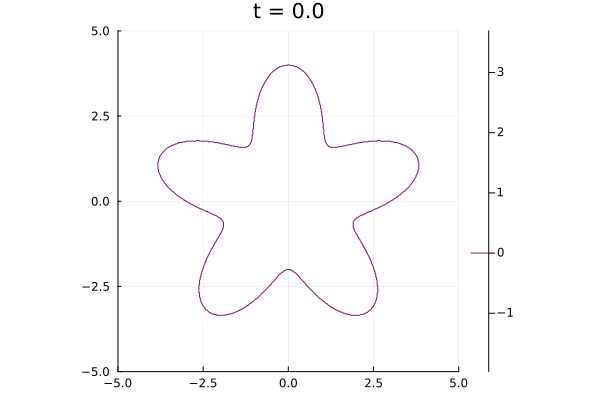}
      %\subcaption{t = 0.0}\label{fig:soliton_0_0}
    \end{minipage}
    \begin{minipage}[b]{0.5\linewidth}
      \centering
      \includegraphics[keepaspectratio, scale=0.4]{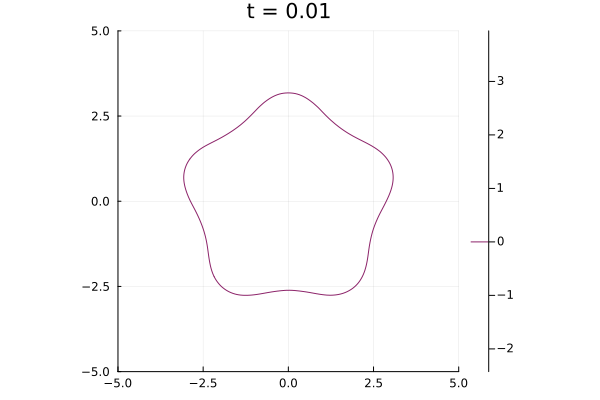}
      %\subcaption{t=0.16}\label{fig:soliton_0_16}
    \end{minipage} \\
    \begin{minipage}[b]{0.5\linewidth}
      \centering
      \includegraphics[keepaspectratio, scale=0.4]{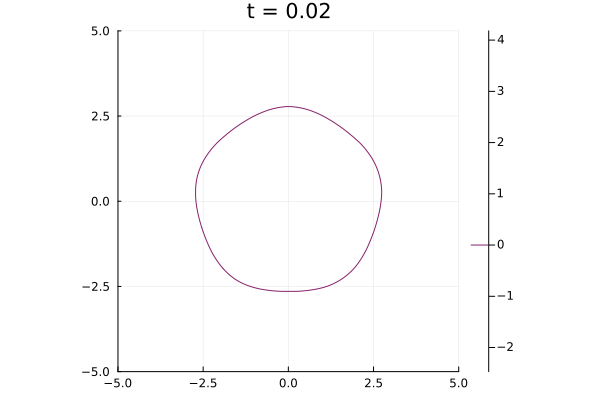}
      %\subcaption{t = 0.21}\label{fig:soliton_0_21}
    \end{minipage}
    \begin{minipage}[b]{0.5\linewidth}
      \centering
      \includegraphics[keepaspectratio, scale=0.4]{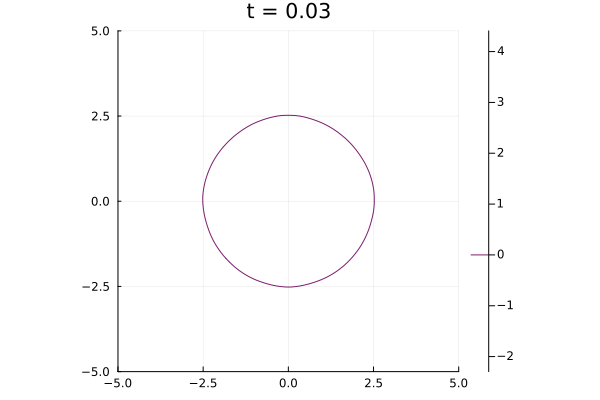}
      %\subcaption{t = 0.31}\label{fig:soliton_0_31}
    \end{minipage}
    \caption{Evolution of star shaped curve}\label{fig:projected}
\end{figure}

\subsubsection{Pi shaped curve}
We borrow the initial data of the pi curve from the website "https://ja.wolframalpha.com/". 
Since its parameterization is quite complicated, we do not cite it. 
We have multiply coordinates by $\frac{1}{240}$ to be included our $\Omega$. 
The hyperparameters used in this case are same as in the previous section, namely star shaped curve.

\begin{figure}[H]
    \begin{minipage}[b]{0.5\linewidth}
      \centering
      \includegraphics[keepaspectratio, scale=0.4]{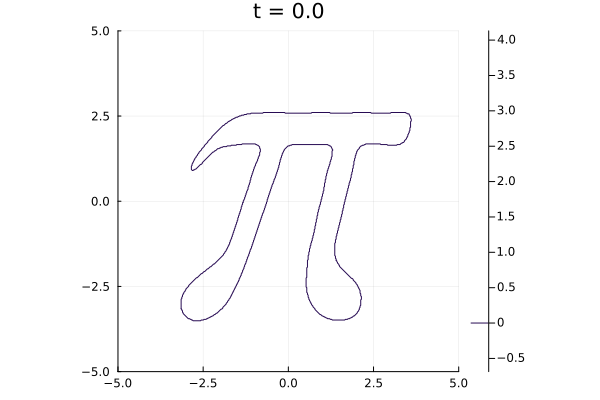}
      %\subcaption{t = 0.0}\label{fig:soliton_0_0}
    \end{minipage}
    \begin{minipage}[b]{0.5\linewidth}
      \centering
      \includegraphics[keepaspectratio, scale=0.4]{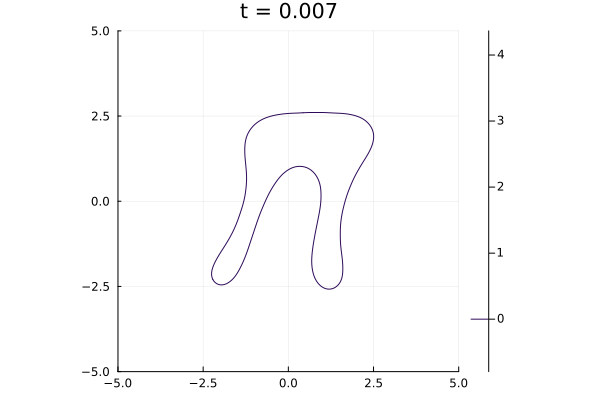}
      %\subcaption{t=0.16}\label{fig:soliton_0_16}
    \end{minipage} \\
    \begin{minipage}[b]{0.5\linewidth}
      \centering
      \includegraphics[keepaspectratio, scale=0.4]{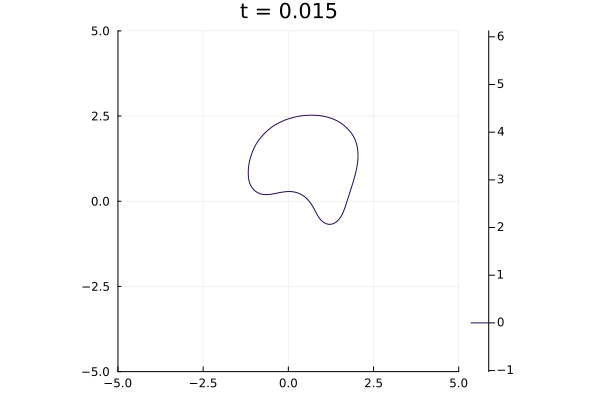}
      %\subcaption{t = 0.21}\label{fig:soliton_0_21}
    \end{minipage}
    \begin{minipage}[b]{0.5\linewidth}
      \centering
      \includegraphics[keepaspectratio, scale=0.4]{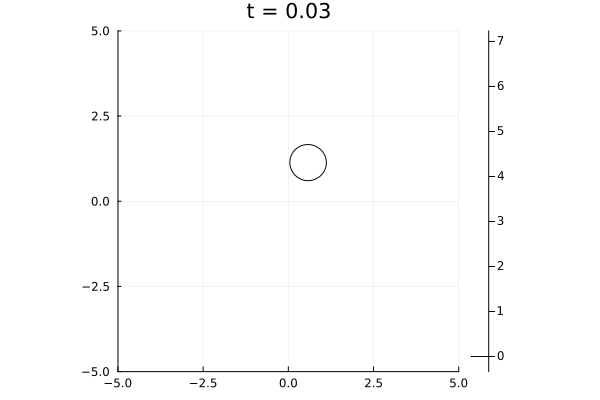}
      %\subcaption{t = 0.31}\label{fig:soliton_0_31}
    \end{minipage}
    \caption{Evolution of Pi shaped curve}\label{fig:projected}
\end{figure}

\section{Split Bregman method for capillary functional}
Let us explore a way to apply the Split Bregman method to our problem which is of the form:
\begin{equation*}
    \min_{u\in L^2(\Omega)}C_\beta(u) + \frac{\mu}{2}\|u - f\|^2_2.
\end{equation*}
Following the idea of the method, we prefer to split this problem into the following two sub-problems:
\begin{equation}\label{eq:minimizing_d_bdd}
    \mathbf{d} := \argmin_{\mathbf{d}}\int_\Omega|\mathbf{d}| + \frac{\lambda}{2}\|\mathbf{d} - \nabla u - b\|^2_2
\end{equation}
and
\begin{equation}\label{eq:minimizing_u_bdd}
    u := \argmin_{u}\int_{\pOmega}\beta\gamma u\,\dH{1} + \frac{\mu}{2}\|u - f\|^2_2 + \frac{\lambda}{2}\|\mathbf{d} - \nabla u - b\|^2_2.
\end{equation}
To this end, we recall the derivation of the equality \eqref{eq:EulerLagrangeInBulkSplitBregmanDiscrete}. This equality is nothing but discrete version of the following equality:
\begin{equation}\label{eq:EulerLagrangeInBulkSplitBregman}
    (\mu I - \lambda\Delta)u = \mu f - \lambda\partial_x(d_x - b_x) - \lambda\partial_y(d_y - b_y).
\end{equation}
The formula \eqref{eq:EulerLagrangeInBulkSplitBregman} is derived as the Euler-Lagrange equation of the energy to be minimized in \eqref{eq:minimizing_u}. Indeed, for any $\varphi\in C^\infty(\Omega)$, we calculate as follows:
\begin{multline*}
    \deriv{\varepsilon}\left(\int_{\pOmega}\beta\gamma(u+\varepsilon\varphi)\,\dH{1} +\frac{\mu}{2}\|u + \varepsilon\varphi - f\|^2_2 + \frac{\lambda}{2}\|\mathbf{d} - \nabla u - \varepsilon\nabla\varphi - \mathbf{b}\|^2_2 \right) = \\
    \mu\int_\Omega(u+\varepsilon\varphi-f)\varphi + \lambda\int_\Omega(d_x-u_x-\varepsilon\varphi_x-b_x)(-\varphi_x) + \lambda\int_\Omega(d_y-u_y-\varepsilon\varphi_y - b_y)(-\varphi_y) = \\
    \mu\int_\Omega(u+\varepsilon\varphi-f)\varphi + \lambda\int_\Omega\left\{\partial_x(d_x-\varepsilon\varphi_{xx}- b_x) - u_{xx}\right\}\varphi + \lambda\int_\Omega\left\{\partial_y(d_y - \varepsilon\varphi_{yy} - b_y) - u_{yy}\right\}\varphi.
\end{multline*}
Here, we have used integration by parts to get the second equality from the first one. 
Evaluating the above equality at $\varepsilon = 0$ and taking into account that $\varphi$ has been arbitrarily taken, we obtain \eqref{eq:EulerLagrangeInBulkSplitBregman}. Now, suppose that $\varphi$ is smooth up to $\pOmega$. Then, we again calculate the variation of the energy \eqref{eq:minimizing_u_bdd}. Since the formula \eqref{eq:EulerLagrangeInBulkSplitBregman} is valid in $\Omega$, integration by parts yields the following formula:
\begin{equation*}
    \int_{\pOmega}\beta\varphi\,\dH{1} - \lambda\int_{\pOmega}(d_x - u_x - b_x)n_x\varphi\,\dH{1} - \lambda\int_{\pOmega}(d_y-u_y-b_y)n_y\varphi\,\dH{1} = 0.
\end{equation*}
Here, $n_x$ and $n_y$ denote the first element and the second one of the outer normal vector to $\pOmega$, respectively. Since $\varphi$ is arbitrary, we have
\begin{equation}\label{eq:EulerLagrangeOnBddSplitBregman}
    \beta = \lambda(d_x - u_x - b_x)n_x + \lambda(d_y - u_y - b_y)n_y\ \ \mbox{on}\ \ \pOmega.
\end{equation}
\begin{rem}
    The formula \eqref{eq:EulerLagrangeOnBddSplitBregman} is somehow rational.   
    Actually, assume that $\Omega$ is a rectangular domain in $\mathbb{R}^2$ and
    $\partial E$ touches $\pOmega$ on the left wall with the angle $\theta$. 
    At the first step of Bregman iteration, $\iter{d}{0}{x} = \iter{b}{0}{x} = 0$ is assumed.
    Moreover, it holds that $n_x=-1$ and $n_y = 0$. Thus, if $\lambda = 1$, then $\beta = u_x$
    follows on the left wall. Because the interface is described as the zero level set of $u$,
    its outward unit normal vector is $\nabla u/|\nabla u|$. As $u$ is closed to 
    the signed distance function $d$, we expect $|\nabla u|\approx 1$. 
    By these observations, it likely holds that $\nabla u/|\nabla u|\cdot\mathbf{n} = \beta$.
    On the other hand, if the contact angle between $\partial E$ and $\pOmega$ is equal to $\theta$,
    we see that $\nabla u/|\nabla u|\cdot\mathbf{n} = -\cos{\theta}$. 
    Hence, $\beta = -\cos{\theta}$ is inferred. Thus, the formula \eqref{eq:EulerLagrangeOnBddSplitBregman} makes sense.
\end{rem}

\begin{rem}
    A similar argument is found in Appendix B \cite{ObermanOsherTakeiTsai2011}.
    Therein, integration by parts was used to obtain the boundary condition of the  
    Euler-Lagrange equation. Although, the boundary integral term did not appear.
    Eq. (3.12) in \cite{ObermanOsherTakeiTsai2011} corresponds to \eqref{eq:EulerLagrangeOnBddSplitBregman}
    when $\beta\equiv 0$.
\end{rem}

The Euler-Lagrange equations up to the boundary have been derived as in \eqref{eq:EulerLagrangeInBulkSplitBregman} and \eqref{eq:EulerLagrangeOnBddSplitBregman}. 
In the sequel, we solve this system numerically and see that the scheme works as expected.
A benchmark function is the solution to the following boundary value problem which was considered in \cite{AltschulerWu1993}:

\begin{equation}\label{eq:SolitonProblem}
    \begin{cases}
        u_t - (\arctan{u_x})_x = 0\ \ \mbox{in}\ \ \Omega\times[0,\infty), \\
        u_x = (\tan{\theta_l})^{-1}\ \ \mbox{on}\ \ \{x = \alpha_x\}\cap\pOmega\times[0,\infty), \\
        u_x = (\tan{\theta_r})^{-1}\ \ \mbox{on}\ \ \{x = \beta_x\}\cap\pOmega\times[0,\infty), \\
        u(0,\cdot) = u_0\in C^\infty(\closure{\Omega}),
    \end{cases}
\end{equation}
where $\theta_l$ denotes the angle between $(0,1)^T$ and $\{x = \alpha_x\}\cap\pOmega$;
$\theta_r$ denotes the angle between $(0,1)^T$ and $\{x = \beta_x\}\cap\pOmega$.
Take a look at the following figure to understand the setting:

\begin{figure}[H]
    \centering
    \includegraphics[keepaspectratio, scale=0.25]{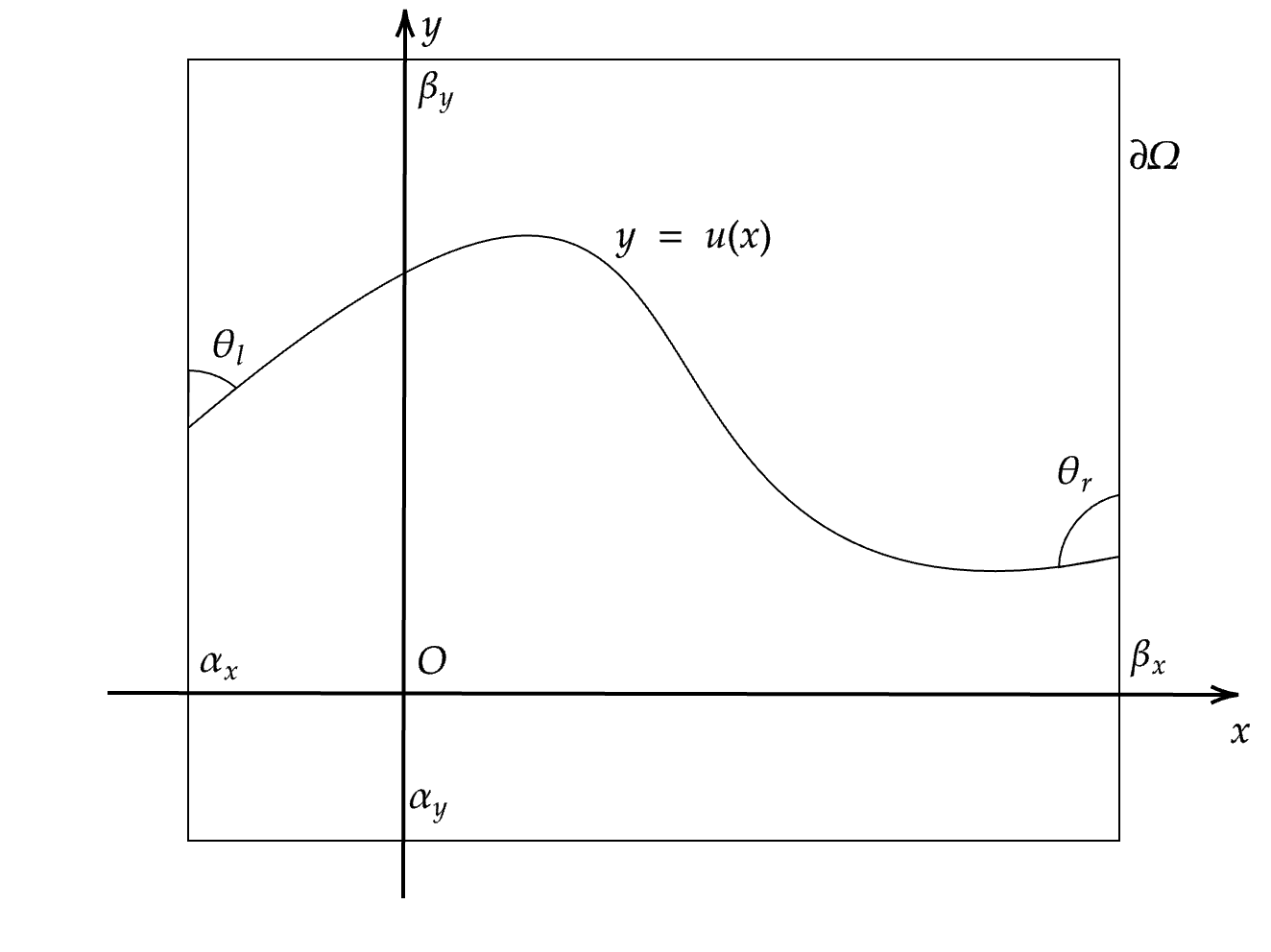}
    \caption{Contact angle problem}\label{fig:contact_angle_pde}
\end{figure}

We shall explain how to treat the boundary part of the system, namely \eqref{eq:EulerLagrangeOnBddSplitBregman}. 
It is convenient to be able to get the discrete values of $u$ on the boundary $\pOmega$ in the formula \eqref{eq:DiscreteEulerLagrange}. 
However, to calculate the quantities $\iter{u}{k}{1,j}, \iter{u}{k}{N_y,j}$ for $1\leq j\leq N_x$ and $\iter{u}{k}{i,1}, \iter{u}{k}{i,N_x}$ for $1\leq i\leq N_y$, we need the values $\iter{u}{k-1}{0,j}, \iter{u}{k-1}{N_y+1,j}, \iter{u}{k-1}{i,1}$ and $\iter{u}{k-1}{i,N_x+1}$ which is fictional. It is a common way to consider that this kind of imaginary points exist and to compute these values in terms of an imposed Neumann boundary condition, that is to say the formula \eqref{eq:EulerLagrangeOnBddSplitBregman}. At first, consider the case where mesh points $(x_j, y_i)$ are on the edge of $\pOmega$. Then, By the central difference method, $\partialDeriv{u}{x}{i}{j}$ and $\partialDeriv{u}{y}{i}{j}$ are computed as follows:
\begin{multline}\label{eq:FictionalCentralDifference}
    \partialDeriv{u}{x}{i}{1} \approx\frac{u_{i,2} - u_{i,0}}{\Delta x},\ \ \partialDeriv{u}{x}{i}{N_x} \approx\frac{u_{i,N_x + 1} - u_{i,N_x-1}}{\Delta x},\\ \partialDeriv{u}{y}{1}{j} \approx \frac{u_{2,j} - u_{0,j}}{\Delta x},\ \ \partialDeriv{u}{y}{N_y}{j} \approx \frac{u_{N_y+1,j} - u_{N_y-1,j}}{\Delta x}.
\end{multline}
On the other hand, the formula \eqref{eq:EulerLagrangeOnBddSplitBregman} yields 
\begin{multline}\label{eq:FictionalBoundaryCondition}
    \partialDeriv{u}{x}{i}{1} = \partialDeriv{d}{x}{i}{1} - \partialDeriv{b}{x}{i}{1} + \frac{1}{\lambda}\beta_{i,1}, \ \ \partialDeriv{u}{x}{i}{N_x} = \partialDeriv{d}{x}{i}{N_x} -\partialDeriv{b}{x}{i}{N_x} - \frac{1}{\lambda}\beta_{i,N_x}, \ \ \\ \partialDeriv{u}{y}{1}{j} = \partialDeriv{d}{y}{1}{j} - \partialDeriv{b}{y}{1}{j} +\frac{1}{\lambda}\beta_{1,j}, \partialDeriv{u}{y}{N_y}{j} = \partialDeriv{d}{y}{N_y}{j} - \partialDeriv{b}{y}{N_y}{j} - \frac{1}{\lambda}\beta_{N_y,j}.
\end{multline}
Combining \eqref{eq:FictionalCentralDifference} and \eqref{eq:FictionalBoundaryCondition} gives
\begin{multline}\label{eq:FictionalEdgeFormula}
    u_{i,0} = u_{i,2} - 2\Delta x(\partialDeriv{d}{x}{i}{1} - \partialDeriv{b}{x}{i}{1} + \frac{1}{\lambda}\beta_{i,1}),\ \ u_{i,N_x+1} = u_{i,N_x-1} + 2\Delta x(\partialDeriv{d}{x}{i}{N_x} - \partialDeriv{b}{x}{i}{N_x} - \frac{1}{\lambda}\beta_{i,N_x}),\\ \ \ u_{0,j} = u_{2,j} - 2\Delta x(\partialDeriv{d}{y}{1}{j} - \partialDeriv{b}{y}{1}{j} + \frac{1}{\lambda}\beta_{1,j}),\ \ u_{N_y+1,j} = u_{N_y-1,j} + 2\Delta (\partialDeriv{d}{y}{N_y}{j} - \partialDeriv{b}{y}{N_y}{j} - \frac{1}{\lambda}\beta_{N_y,j}).
\end{multline}
Here, to derive \eqref{eq:FictionalBoundaryCondition}, we note that $\partialDeriv{n}{x}{i}{1} = \partialDeriv{n}{y}{1}{j} = -1$, $\partialDeriv{n}{x}{i}{N_x} = \partialDeriv{n}{y}{N_y}{j} = 1$ and $\partialDeriv{n}{x}{1}{j} = \partialDeriv{n}{x}{N_y}{j} = \partialDeriv{n}{y}{i}{1} = \partialDeriv{n}{y}{i}{N_x} = 0$ for $2\leq i\leq N_y-1$ and $2\leq j\leq N_x-1$. Secondly, let us obtain the discrete values at the corner points $(1,1),(1,N_x),(N_y,1)$ and $(N_y,N_x)$. Since $\pOmega$ is not smooth at these points, we artificially assume that $\mathbf{n}_{1,1} = (-1/\sqrt{2},-1/\sqrt{2}),\mathbf{n}_{1,N_x} = (1/\sqrt{2},-1/\sqrt{2}),\mathbf{n}_{N_y,1} = (-1/\sqrt{2},1/\sqrt{2})$ and $\mathbf{n}_{N_y,N_x} = (1/\sqrt{2},1/\sqrt{2})$. Then, by a similar argument as above, we have
\begin{multline}\label{eq:FictionalCornerFormula}
    u_{N_y,0} + u_{N_y+1, 1} = \\u_{N_y,2} + u_{N_y-1,1} + 2\Delta x\left\{-\partialDeriv{d}{x}{N_y}{1} + \partialDeriv{b}{x}{N_y}{1} + \partialDeriv{d}{y}{N_y}{1} - \partialDeriv{b}{y}{N_y}{1} - \frac{\sqrt{2}}{\lambda}\beta_{N_y,1}\right\},\\
    u_{N_y,N_x+1} + u_{N_y+1,N_x} = \\u_{N_y,N_x-1} + u_{N_y-1,N_x} + 2\Delta x\left\{\partialDeriv{d}{x}{N_y}{N_x} - \partialDeriv{b}{x}{N_y}{N_x} + \partialDeriv{d}{y}{N_y}{N_x} - \partialDeriv{b}{y}{N_y}{N_x} -\frac{\sqrt{2}}{\lambda}\beta_{N_y,N_x}\right\},\\
    u_{0,N_x} + u_{1,N_x+1} =\\ u_{2,N_x}+u_{1,N_x-1} + 2\Delta x\left\{\partialDeriv{d}{x}{1}{N_x} - \partialDeriv{b}{x}{1}{N_x} - \partialDeriv{d}{y}{1}{N_x} + \partialDeriv{b}{y}{1}{N_x} -\frac{\sqrt{2}}{\lambda}\beta_{1,N_x}\right\},\\
    u_{0,1} + u_{1,0} = u_{2,1} + u_{1,2} - 2\Delta x\left\{\partialDeriv{d}{x}{1}{1} - \partialDeriv{b}{x}{1}{1} + \partialDeriv{d}{y}{1}{1} + \partialDeriv{b}{y}{1}{1} + \frac{\sqrt{2}}{\lambda}\beta_{1,1}\right\}.
\end{multline}
We substitute the formulae \eqref{eq:FictionalEdgeFormula} and \eqref{eq:FictionalCornerFormula} into \eqref{eq:DiscreteEulerLagrange} if we encounter needs for computing boundary points on $\pOmega$. 
In this way, we carry out numerical experiment for the boundary contact case. 
As an initial data, we select the curve whose zero level set is the graph of:

\begin{equation*}
    u(x) := \frac{4}{\pi}\log{\left|\cos{\left(-\frac{\pi}{4}x + \frac{\pi}{4}\right)}\right|} + \frac{1}{2} + \frac{2}{\pi}\log{2}\ \ \mbox{for}\ \ x\in[0,2].
\end{equation*}

Observe that the graph of $u$ does not change its shape when it is evolved by
curvature. This kind of curves is often called the Grim Reaper 
or the translating soliton in the literature.
The hyperparameters are set as $\alpha_x := 0.0, \beta_x := 2.0, \alpha_y := 0.0, \beta_y := 1.0, N_x := 800, N_y := 400, h := 0.5\Delta x^2, \mu := 1/h$ and $\lambda := 1.0$. $\beta$ is defined as follows: 
$\beta_{i,1} = \beta_{i,N_x} = -1/\sqrt{2}$ for $1\leq i\leq N_y$, $\beta_{1,j} = 1/\sqrt{2}$ for $2\leq j \leq N_x-1$ and $\beta_{N_y,j} = 0.0$ for $2\leq j\leq N_x-1$. 
Note that the setting of $\beta$ implies the interface (in our case, it is the zero level set of $u$) should intersect $\pOmega$ 
with the angle $\pi / 4$. See the following figures and confirm that the shape is totally preserved and moves downward.

\begin{figure}[H]
    \begin{minipage}[b]{0.5\linewidth}
      \centering
      \includegraphics[keepaspectratio, scale=0.35]{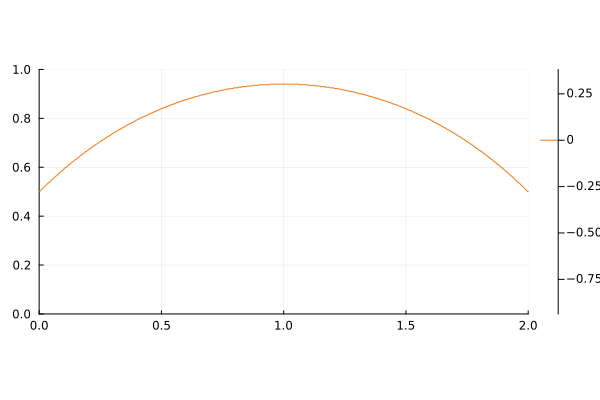}
      %\subcaption{t = 0.0}\label{fig:soliton_0_0}
    \end{minipage}
    \begin{minipage}[b]{0.5\linewidth}
      \centering
      \includegraphics[keepaspectratio, scale=0.35]{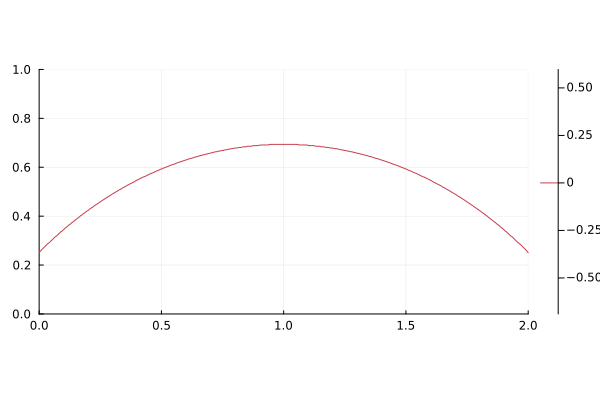}
      %\subcaption{t=0.16}\label{fig:soliton_0_16}
    \end{minipage} \\
    \begin{minipage}[b]{0.5\linewidth}
      \centering
      \includegraphics[keepaspectratio, scale=0.35]{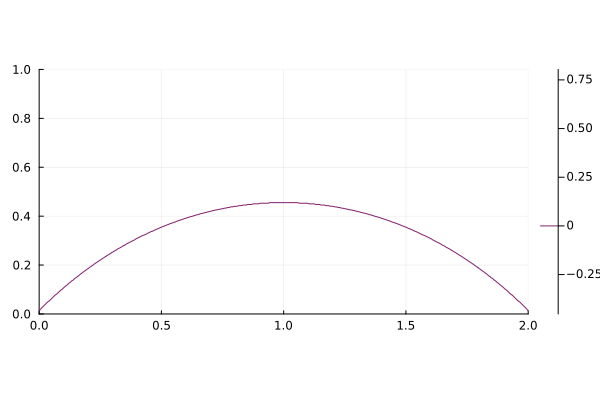}
      %\subcaption{t = 0.21}\label{fig:soliton_0_21}
    \end{minipage}
    \begin{minipage}[b]{0.\linewidth}
      \centering
      \includegraphics[keepaspectratio, scale=0.35]{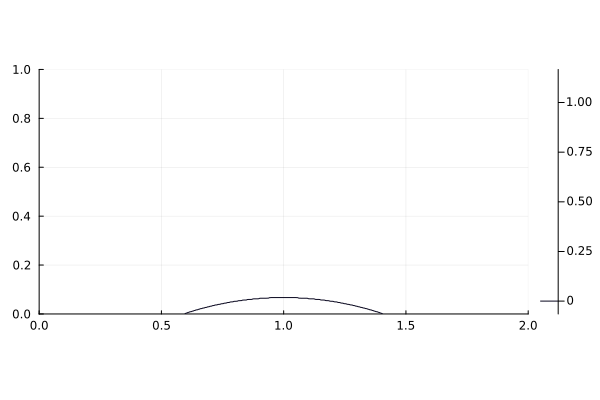}
      %\subcaption{t = 0.31}\label{fig:soliton_0_31}
    \end{minipage}
    \caption{Evolution of translating soliton}\label{fig:projected}
\end{figure}

According to Corollary 1.3 \cite{AltschulerWu1993}, if $\theta_l + \theta_r = \pi$, then
the evolution of the curve converges to a straight line as $t\to\infty$.
To confirm this fact in the numerical experiment, we choose the initial data as:

\begin{equation*}
    u(x) := \frac{1}{4}\sin{(\pi x)}+ \frac{1}{2}\ \ \mbox{for}\ \ x\in[0,2].
\end{equation*}

The hyperparameters are set as $\alpha_x := 0.0, \beta_x := 2.0, \alpha_y := -2.0, \beta_y := 1.0, N_x := 300, N_y := 450, h := 0.5\Delta x^2, \mu := 1/h$ and $\lambda := 1.0$.
$\beta$ is defined as follows: 

\begin{figure}[H]
    \begin{minipage}[b]{0.5\linewidth}
      \centering
      \includegraphics[keepaspectratio, scale=0.25]{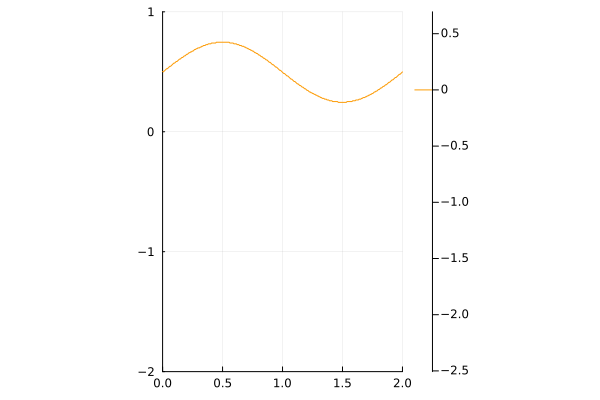}
      %\subcaption{t = 0.0}\label{fig:soliton_0_0}
    \end{minipage}
    \begin{minipage}[b]{0.5\linewidth}
      \centering
      \includegraphics[keepaspectratio, scale=0.25]{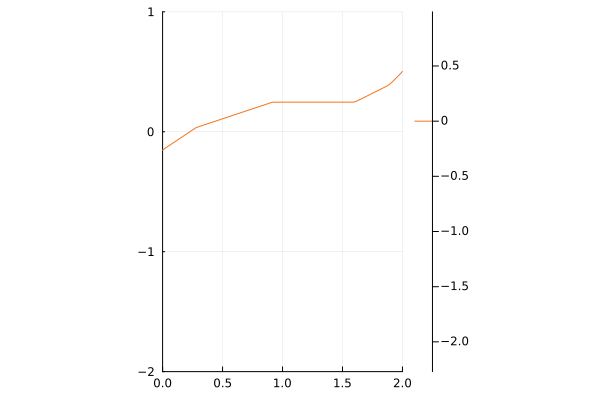}
      %\subcaption{t=0.16}\label{fig:soliton_0_16}
    \end{minipage}
    \begin{minipage}[b]{0.5\linewidth}
      \centering
      \includegraphics[keepaspectratio, scale=0.25]{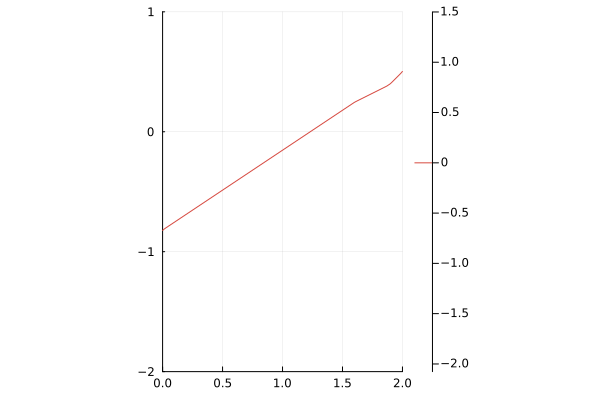}
      %\subcaption{t = 0.21}\label{fig:soliton_0_21}
    \end{minipage}
    \begin{minipage}[b]{0.5\linewidth}
      \centering
      \includegraphics[keepaspectratio, scale=0.25]{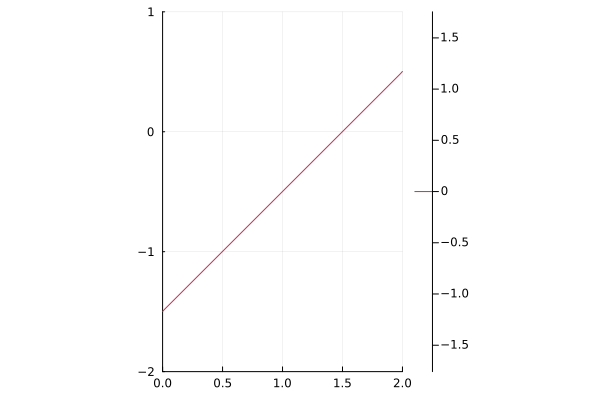}
      %\subcaption{t = 0.31}\label{fig:soliton_0_31}
    \end{minipage}
    \caption{Convergence to a straight line}\label{fig:projected}
\end{figure}
\section{Acknowledgements}
The authors are grateful to Professor Harald Garcke 
for his valuable comments for the problem setting of the prescribed contact angle condition.
The work of the second author was partly supported by the Japan Society 
for the Promotion of Science (JSPS) through the grants Kakenhi: 
No.~19H00639, No.~18H05323, No.~17H01091, and by Arithmer Inc.\ and Daikin Industries, Ltd.\ 
through collaborative grants.

\bibliography{cite}
\bibliographystyle{is-plain}
\end{document}